\documentclass{amsart}
\usepackage{amsmath,amsfonts,amssymb,amsthm,amscd,comment,euscript}
\usepackage[all]{xy}
\usepackage{graphicx}
\usepackage{enumerate} %allows to change enumeration items easily
\usepackage[colorlinks=true]{hyperref} %to generate pdf with links
\usepackage[usenames,dvipsnames]{xcolor}
\usepackage{enumitem}

\swapnumbers
\theoremstyle{plain}
\newtheorem{lem}{Lemma}[section]
\newtheorem{cor}[lem]{Corollary}
\newtheorem{prop}[lem]{Proposition}
\newtheorem{thm}[lem]{Theorem}

\theoremstyle{definition}

\newtheorem{rem}[lem]{Remark}
\newtheorem{dfn}[lem]{Definition}
\newtheorem{ntt}[lem]{Notation}

\newcommand{\Spec}{\operatorname{Spec}}   % Spectrum
\newcommand{\Sm}{\text{\bf{Sm}}}   % Spectrum

\newcommand{\Os}{\mathcal{O}}

\newcommand{\PP}{\mathbb{P}}
\newcommand{\A}{\mathbb{A}}
\newcommand{\gSL}{\operatorname{\mathbf{SL}}}
\newcommand{\Ms}{M}
\newcommand{\Gm}{\mathbb{G}_m}
\newcommand{\la}{\langle}
\newcommand{\ra}{\rangle}

\newcommand{\ZZ}{\mathbb{Z}}

\newcommand{\coker}{\operatorname{coker}}
\newcommand{\colim}{\operatorname{colim}}
\newcommand{\supp}{\operatorname{supp}}
\newcommand{\divv}{\operatorname{div}}
\newcommand{\charr}{\operatorname{char }}

\begin{document}
\begin{abstract}
Given a field $k$ of characteristic zero and $n\geqslant 0$, we
prove that $H_0(\mathbb{Z}F(\Delta_k^{\bullet},\Gm^{\wedge
n}))=K_n^{MW}(k)$, where $\mathbb{Z}F_{*}(k)$ is the category of
linear framed correspondences of algebraic $k$-varieties, introduced
by Garkusha--Panin in~\cite{GPLinear}
(see~\cite[\S7]{GPFramedMotives} as well), and $K_*^{MW}(k)$ is the
Milnor--Witt $K$-theory of the base field $k$.
\end{abstract}

\title{Framed correspondences and the Milnor--Witt K-theory}
\author{Alexander Neshitov}
\address{Alexander Neshitov, Steklov Mathematical Institute,  St.Petersburg}
\email{alexander.neshitov@gmail.com}
\thanks{The work was supported by the Russian Science Foundation grant 14-11-00456}
\maketitle

\section{Introduction}
The theory of framed correspondences and framed (pre)sheaves was
developed by Voevodsky in~\cite{VoevUnpublished}. Using this theory,
Garkusha and Panin introduce and study in~\cite{GPFramedMotives}
framed motives of algebraic varieties. As an application, they
construct an explicit fibrant replacement of the suspension
bispectrum $\Sigma_{S^1}^{\infty}\Sigma_{\Gm}^{\infty}X_{+}$ of a
smooth algebraic variety $X\in\Sm_k$. As a consequence, this allows
to reduce the computation of the motivic cohomotopy groups
$\pi^{n,n}(\Sigma_{S^1}^{\infty}\Sigma_{\Gm}^{\infty}S^0)(k)$ to the
computation of the zeroth homology group
$H_0(\mathbb{Z}F(\Delta_k^{\bullet},\Gm^{\wedge n}))$ of an explicit
simplicial abelian group $\mathbb{Z}F(\Delta^{\bullet}_k,\Gm^{\wedge
n})$ (see~\cite[Corollary 10.7]{GPFramedMotives}).

Let $k$ be a field of characteristic zero. We show in this paper
that $H_0(\mathbb{Z}F(\Delta_k^{\bullet},\Gm^{\wedge n}))$ is the
$n$th Milnor--Witt group $K_n^{MW}(k)$ of $k$. This result together
with~\cite[Corollary 10.7]{GPFramedMotives} recovers the celebrated
theorem of Morel~\cite[Theorem 5.40]{Morel} for such fields.

\subsection*{Framed correspondences}

Here we briefly sketch the definition of framed correspondences. We
refer the reader to~\cite[\S2]{GPFramedMotives} for more details.

\begin{dfn}
Let $U$ be a smooth scheme and $Z\subseteq U$ a closed subset of codimension $n$. The set of regular functions $\phi_i,\ldots,\phi_n\in k[U]$ is called a framing of $Z$ in $U$ if $Z$ coincides with the closed subset defined by equations $\phi_1=0,\ldots,\phi_n=0.$
\end{dfn}
\begin{dfn}
An explicit framed correspondence $c=(U,\phi,f)$ of level $n$ from $X$ to $Y$ consists of the following data:
\begin{itemize}
\item a closed subset $Z$ of $\A^n_X$ which is finite over $X$,
\item an \'etale neighborhood $\alpha\colon U\to\A^n_X$ of $Z$,
\item a framing $\phi=(\phi_1,\ldots,\phi_n)$ of $Z$ in $U$,
\item a morphism $f\colon U\to Y.$
\end{itemize}
\end{dfn}
The subset $Z$, also denoted by $\supp(c)$, is called the support of
the explicit correspondence $c=(U,\phi,f)$. Two explicit level $n$
correspondences $(U,\phi,f)$ and $(U',\phi',f')$ are equivalent if
their supports coincide and in some open neighborhood of $Z$ in
$U\times_{\A^n_X}U'$ one has $\phi\circ pr=\phi'\circ pr'$ and
$f\circ pr=f'\circ pr'$. By definition, a framed correspondence of
level $n$ is an equivalence class of explicit framed correspondences
of level $n$. We shall write $Fr_n(X,Y)$ to denote the set of framed
correspondences of level $n$ from $X$ to $Y$.

There is a well defined composition~\cite[\S2]{GPFramedMotives} of
correspondences:
\[Fr_n(X,Y)\times Fr_m(Y,Z)\to Fr_{n+m}(X,Z).\]
The composition of $(U,\phi,f)\in Fr_n(X,Y)$ and $(U',\phi',f')\in
Fr_m(X,Y)$ is given by the triple $(U\times_Y U',(\phi\circ
pr_U,\phi'\circ pr_{U'}),f'\circ pr_{U'})$.

\subsection*{Linear Framed Correspondences}
Here we sketch the construction of the category of linear framed
correspondences $\mathbb{Z}F_*(k)$ introduced in~\cite{GPLinear}
(see also~\cite[\S7]{GPFramedMotives}).
\begin{dfn}\label{13}
Given $n\geqslant 0$ we define $\ZZ F_n(X,Y)$ as a quotient of the
free abelian group $\ZZ Fr_n(X,Y)$ generated by $Fr_n(X,Y)$. Namely
one sets,
\[
\ZZ F_n(X,Y)=\ZZ Fr_n(X,Y)/\langle c-c_1-c_2\mid
\supp(c)=\supp(c_1)\sqcup\supp(c_2) \rangle.
\]
\end{dfn}
Note that $\ZZ F_n(X,Y)$ is a free abelian group generated by correspondences with connected support. The composition of correspondences descends to a composition 
\[\ZZ F_n(X,Y)\times\ZZ F_m(Y,Z)\to \ZZ F_{n+m}(X,Z).\]
\begin{dfn}
The category of linear framed correspondences $\mathbb{Z}F_*(k)$ is
an additive category whose objects are those of $\Sm_k$ and for
every $X,Y\in\Sm_k$
\[Hom_{\ZZ F_*(k)}(X,Y)=\oplus_{n\geqslant 0}\ZZ F_n(X,Y).\]
\end{dfn}

For any $Y\in \Sm_k$ consider a correspondence $\sigma_Y=(\A^1_Y,id,pr_Y)\in Fr_1(Y,Y)$. Note that for any $c=(U,\phi,f)\in Fr_n(X,Y)$ the composition $\sigma_Y\circ c=(U\times_k\A^1,(\phi,id_{\A^1}),f\circ pr_U)\in Fr_{n+1}(X,Y).$

For any $X,Y\in\Sm_k$ we set
\[\ZZ F(X,Y)=\colim(\ZZ F_*(X,Y)\stackrel{\sigma_Y}\to\ZZ F_{*+1}(X,Y)\stackrel{\sigma_Y}\to\ldots).\]

We consider the simplicial abelian group $\ZZ
F(\Delta^{\bullet}_k\times X,Y)$ where $\Delta^{\bullet}_k$ is the standard
cosimplicial object in $\Sm_k$ defined as follows: $\Delta^n_k$ is the
hyperplane defined by the equation $x_0+\ldots+x_n=1$ in $\A^{n+1}_k$
with standard coface and codegeneracy maps.
In particular we will be interested in its zeroth homology group
\[H_0(\ZZ F(\Delta^{\bullet}_k\times X,Y))=\coker (\ZZ F(\A^1_k\times X,Y)\stackrel{i_0^*-i_1^*}\to\ZZ F(X,Y)).\]

\begin{rem}
Since $\sigma_Y$ does not commute with the composition map, there is
no well-defined composition $\ZZ F(X,Y)\times \ZZ F(Y,Z)\to \ZZ
F(X,Z)$. Nevertheless applying the simplicial resolution, we get that the composition defined in~\ref{13} descends to a well-defined composition
\[H_0(\ZZ F(\Delta^{\bullet}_k\times X,Y))\times H_0(\ZZ F(\Delta^{\bullet}_k\times Y,Z))\to H_0(\ZZ F(\Delta^{\bullet}_k\times X,Z)).\]
\end{rem}

Let $\Gm$ be the pointed scheme $(\A^1\setminus\{0\},1)$.
For $n\geqslant 0$ and $1\leqslant i\leqslant n$ let $f_i\colon\Gm^{n-1}\to\Gm^n$ be an inclusion given by $f_i(x_1,\ldots,x_{n-1})=(x_1,\ldots,x_{i-1},1,x_{i},\ldots x_{n-1})$. Define the simplicial abelian group
\[\ZZ F(\Delta^{\bullet}_k,\Gm^{\wedge n})=\coker (\oplus_{i=1}^n\ZZ F(\Delta^{\bullet}_k,\Gm^{n-1})\stackrel{f_{i*}}\to \ZZ F(\Delta^{\bullet}_k,\Gm^{n})).\]

The main result of the paper states that
\[H_0(\ZZ F(\Delta^{\bullet}_k,\Gm^{\wedge n}))=K_n^{MW}(k).\]
For brevity we will say that two framed correspondences $c,c'\in
Fr_*(k,Y)$ are equivalent (and write $c\sim c'$) if their classes
coincide in $H_0(\ZZ F(\Delta^{\bullet}_k,Y)).$

The paper is organized as follows: The
section~\ref{preliminary_lemmas} contains some auxiliary results. In
the section~\ref{MultStruct} there is constructed the multiplication
which endows $H_0(\ZZ F(\Delta^{\bullet}_k,\Gm^{\wedge *}))$ with a
structure of a graded ring and a transfer morphism $tr_{L/k}\colon
H_0(\ZZ F(\Delta^{\bullet}_L,\Gm^{\wedge *}))\to H_0(\ZZ
F(\Delta^{\bullet}_k,\Gm^{\wedge *}))$ for finite field extensions
$k\subseteq L.$ Section~\ref{MovingSection}
and~\ref{LevelOneSection} are devoted to some reduction results
needed in the proof of the main result. In the section~\ref{cortok} the construction of the map
$\Phi_*\colon H_0(\ZZ F(\Delta^{\bullet}_k,\Gm^{\wedge *}))\to
K_*^{MW}(k)$ is given. In the section~\ref{m0Section} we prove that
$\Phi_0$ is a ring isomorphism. The section~\ref{ktocor} gives the
construction of the opposite direction map $\Psi_*\colon
K_*^{MW}(k)\to H_0(\ZZ F(\Delta^{\bullet}_k,\Gm^{\wedge *})).$ In
the section~\ref{IsomorphismSection} we check that $\Phi_*$ and
$\Psi_*$ are mutually inverse graded ring isomorphisms.

\section*{Acknowledgement}
I would like to express my sincere gratitude to Professor Ivan Panin
for statement of the problem, constant attention to this work and
numerous discussions concerning the subject of this paper.

\section{Preliminary lemmas}\label{preliminary_lemmas}
For a smooth $Y$ we will need the following facts about
\[
H_0(\ZZ F(\Delta^{\bullet}_k,Y))=\coker (\ZZ F(\A^1_k,Y)\stackrel{i_0^*-i_1^*}\to\ZZ F(\Spec k,Y))
\]

\begin{lem}\label{retract}
Suppose $c=(U,\phi,f)$ represents a framed correspondence in $Fr_*(\Spec k,Y)$ such that $Z=supp(c)$ a single point.
Then there is an equivalent framed correspondence $c'=(U'\phi',f')$ such that there is a projection $U'\to\Spec k(Z)$.
\end{lem}

\begin{proof}
There is a projection from the Henselization $\Spec \Os_{\A^n,Z}^{hens}\to\Spec k(Z).$ Therefore there is an \'etale neighborhood $U'$ of $Z$ with a projection $U'\to\Spec k(Z)$ and an \'etale map $\beta\colon U'\to U$. Then $c$ is equivalent to $(U'\phi\circ\beta,f\circ\beta).$
\end{proof}

\begin{lem}\label{SL}
Let $c=(U,\phi,f)$ represent a framed correspondence in $Fr_n(\Spec k,Y)$. Let $L$ be a field and assume there is a regular map $U\to\Spec L$. Let $A\in\gSL_n(L)$. Then the classes of $c$ and $c'=(U,A\cdot\phi,f)$ coincide in $H_0(\ZZ F(\Delta^{\bullet}_k,Y))$
\end{lem}
\begin{proof}
Since $\gSL_n$ is simply-connected, $\gSL_n(L)$ is generated by $U_{\alpha}(L)$ where $U_{\alpha}$ denote the root subgroups. For any $A_1\in U_{\alpha}(L)$ there us a rational curve $A(t)\colon\A^1_{L}\to U_{\alpha}$ such that $A(0)=1$ and $A(1)=A_1.$ Note that the data $d=(U\times\A^1,A(t)\cdot\phi,f\circ pr_U)$ define a correspondence in $Fr_n(\A^1,Y)$ with $\supp(d)=Z\times\A^1$. Since $i_0^*(d)=(U,\phi,f)$ and $i_1^*(d)=(U,A_1\cdot\phi,f)$, the lemma follows.
\end{proof}

\begin{lem}\label{nilpotent}
Let $(U,\phi,f)$ represent a correspondence in $Fr_n(\Spec k,Y)$. Assume that $U$ is affine. Let $\delta\in k[U]$ be a regular function such that $\delta+(\phi_2,\ldots \phi_n)$ is nilpotent in the quotient ring $k[U]/(\phi_2,\ldots,\phi_n)$.
Then the classes of correspondences $c=(U,\phi,f)$ and $c'=(U,(\phi_1+\delta,\phi_2,\ldots,\phi_n),f)$ coincide in $H_0(\ZZ F(\Delta^{\bullet}_k,Y))$.
\end{lem}

\begin{proof}
Note that the radicals of the ideals $(\phi_1+\delta,\ldots\phi_n)$ and $(\phi_1,\ldots,\phi_n)$ coincide in $k[U]$, so the data $(U,\phi_1+\delta,\ldots, \phi_n,f)$ defines a correspondence in $Fr_n(\Spec k,Y)$.

Let us check that the data $d=(U\times\A^1,(\phi_1+\delta t,\phi_2,\ldots, \phi_n),f\circ pr_U)$ defines a framed correspondence in $Fr_n(\A^1,Y)$.

First we need to check that the support of $(\phi_1+\delta t,\phi_2,\ldots, \phi_n)$ is finite over $\A^1$. i.e.
$k[U][t]/(\phi_1+\delta t,\phi_2,\ldots,\phi_n)$ is a finite $k[t]$-module. Note that for any $m$ we have that $\phi_1^m-(\delta t)^m$ is divisible by $\phi_1-\delta t$, so there is an inclusion of ideals
\[(\phi_1^m-(\delta t)^m,\phi_2,\ldots, \phi_n)\subseteq (\phi_1-\delta t,\phi_2,\ldots,\phi_n)\]
By assumption, there is $m$ such that $\delta^m\in(\phi_2\ldots \phi_n).$ Then the ideals $(\phi_1^m-(\delta t)^m,\phi_2,\ldots, \phi_n)$ and $(\phi_1^m,\phi_2,\ldots, \phi_n)$ coincide in $k[U][t]$. Thus there is a surjection
\[
k[U][t]/(\phi_1^m,\phi_2,\ldots, \phi_n)\to k[U][t](\phi_1-\delta t,\phi_2,\ldots,\phi_n)\to 0
\]
Since $k[U]/(\phi_1,\ldots,\phi_n)$ is finite over $k$, then $k[U]/(\phi_1^m,\ldots,\phi_n)$ is finite over $k$, hence $k[U][t]/(\phi_1^m,\phi_2,\ldots, \phi_n)$ is finite over $k[t]$. Then $k[U][t](\phi_1-\delta t,\phi_2,\ldots,\phi_n)$ is also finite over $k[t]$.

Second we need to check that the equations $(\phi_1+\delta t,\phi_2,\ldots,\phi_n)$ define a closed subset in $\A^n\times\A^1$, i.e. the composition $k[\A^n][t]\to k[U][t]\to k[U][t]/(\phi_1+\delta t,\phi_2,\ldots,\phi_n)$ is surjective.
Again, we take $m$ such that $k[U][t]/(\phi_1^m,\phi_2,\ldots, \phi_n)$ covers $k[U][t](\phi_1-\delta t,\phi_2,\ldots,\phi_n)$. Then the equations $(\phi_1^m,\phi_2,\ldots, \phi_n)$ have the same support as $\phi$, therefore the composition $k[\A^n]\to k[U]\to k[U]/(\phi_1^m,\phi_2,\ldots, \phi_n)$ is surjective, hence the composition $k[\A^n][t]\to k[U][t]/(\phi_1+\delta t,\phi_2,\ldots,\phi_n)$ is surjective.

Then the data $d$ defines a correspondence in $Fr_n(\A^1,Y)$ with $i_0^*(d)=c$ and $i_1^*(d)=c'.$
\end{proof}

\begin{lem}\label{multiplication}
Let $c=(U,\phi,f)$ represent a framed correspondence in $Fr_n(\Spec k,Y)$ with $U$ affine and $Z=\supp(c)$ is a single point.
Denote by $\Ms_{Z,U}$ the maximal ideal of $k[U]$ defining $Z$. Suppose $m\in\Ms_{Z,U}$

Then there is an open neighborhood $U'$ of $Z$ in $U$ such that $c$ is equivalent to the correspondence
$c'=(U',(\phi_1(1+m),\phi_2,\ldots\phi_n),f)$ in $H_0(\ZZ F(\Delta^{\bullet}_k,Y))$.
\end{lem}
\begin{proof}
Let us check that the data $d=(U\times\A^1\setminus Z(1+mt),((1+mt)\phi_1,\phi_2,\ldots\phi_n),f\circ pr_U)$ define
a correspondence in $Fr_n(\A^1,k).$

First, we check that the support of equations $\phi'=((1+mt)\phi_1,\phi_2,\ldots\phi_n)$ is finite over $\A^1$, i.e. the quotient ring $k[U][t,(1+mt)^{-1}]/((1+mt)\phi_1,\phi_2,\ldots,\phi_n)$ is finite over $k[t]$. Note that
\[
k[U][t]_{1+mt}/((1+mt)\phi_1,\phi_2,\ldots,\phi_n)=k[U][t]_{1+mt}/(\phi_1,\phi_2,\ldots\phi_n)
\]
Thus $\phi'$ defines an open subset of  $Z\times\A^1$ defined be the $1+mt\neq 0$.
Show that this open subset coincides with the whole $Z\times\A^1,$ i.e. the ideal $I=(\phi_1,\ldots,\phi_n)+(1+mt)$ equals $(1)$ in $k[U][t]$. Note that the radical of the ideal $(\phi)$ coincides with $M_Z$: $\sqrt{(\phi_1,\ldots\phi_n)}=\Ms_{Z,U},$ so $m^r\in(\phi_1,\ldots\phi_n)$ for big enough $r$.
We may assume that $r$ is odd, so $1+(mt)^r$ is divisible by $(1+mt).$ Then $1+(mt)^r\in (\phi_1,\ldots,\phi_n)+(1+mt)$. Since $m^rt^r\in(\phi_1,\ldots,\phi_n)$ we get that $1\in I$.

As we have checked, the equations $\phi'$ define the closed subset $Z\times\A^1$ in $\A^n\times\A^1$. Then $d$ defines an element in $Fr_n(\A^1,Y)$ and $i_0^*(d)=c$ and $i_1^*(d)=c'.$

\end{proof}

\begin{lem}\label{ratroots}
Suppose $p(x),q(x)\in k[x]$ are degree $n$ polynomial with the same leading term. Then two correspondences
$c=(\A^1,p(x),pr_k)$ and $c'=(\A^1,q(x),pr_k)$ are equivalent in $H_0(\ZZ F(\Delta^{\bullet}_k,\Spec k)).$

Moreover, if $p(x)$ has separable rational roots $\alpha_1,\ldots,\alpha_n$,
then the correspondence $(\A^1,p(x),pr_k)$ is equivalent to the sum $\sum_{i=1}^n(\A^1,p'(\alpha_i)x,pr_k)$.
\end{lem}
\begin{proof}
Since $p$ and $q$ have the same leading term, we have that $\deg(q-p)<n.$ Then for the data
$d=(\A^1\times\A^1,p(x)+(q(x)-p(x))t,pr_k)$ $\supp(d)$ is finite over $\Spec k[t]$, so $d$ defines an element of $Fr_1(\A^1,k)$ giving an equivalence between $c$ and $c'$.

When all roots of $p$ are separable and rational note that $(\A^1,p(x),pr_k)=\sum_{i=1}^n (U_i,p(x),pr_k)$ in $\ZZ F_1(\Spec k,\Spec k)$ where $U_i$ is an open neighborhood of the root $\alpha_i$ that does not contain other roots. Since $p(x)=p'(x_i)(x-x_i)+\delta_i$ where $\delta_i\in ((x-x_i)^2)$, by lemma~\ref{nilpotent} we have that $(U_i,p(x),pr_k)\sim (U_i,p'(x_i)(x-x_i),pr_k)=(\A^1,p'(x_i)(x-x_i),pr_k).$ Finally note that $\A^1,p'(x_i)(x-x_i),pr_k)\sim (\A^1,p'(x_i)x,pr_k)$ by means of elementary homotopy $(\A^1\times\A^1,p'(x_i)(x-tx_i),pr_k).$
\end{proof}

\begin{lem}\label{deform}
Suppose that $p(x),q(x)\in k[x]$ are degree $n$ polynomials with the same leading term, such that $p(0)=q(0)\neq 0$. Then the correspondences $(\Gm,x^dp(x),id)$ and $(\Gm,x^dq(x),id)$ in $Fr_1(\Spec k,\Gm)$ are equivalent in $H_0(\ZZ F(\Delta^{\bullet}_k,\Gm))$ for any integer $d\geqslant 0.$
\end{lem}
\begin{proof}
Note that $p(x)=ax^n+xp_1(x)+b$ and $q(x)=ax^n+xq_1(x)+b$ for some polynomials $p_1,q_1\in k[x]$ such that $\deg p_1,\deg q_1\leqslant n-2$. Then the equation $ax^n+x(t p_1(x)+(1-t)q_1(x))+b=0$ defines a closed subset $Z$ $\A^1\times\A^1=\Spec k[x,t]$ which is finite over $\Spec k[t]$. Moreover, $Z$ is contained in
$\Spec k[x,x^{-1},t]=\Gm\times\A^1$. Then the data $d=(\Gm\times\A^1,x^d(ax^n+x(t p_1(x)+(1-t)q_1(x))+b),pr_{\Gm})$ defines a correspondence in $Fr_1(\A^1,\Gm)$ which gives a desired equivalence.
\end{proof}

\begin{lem}\label{addone}
Let $c=(U,\phi,f)\in Fr_n(k,Y)$ represent a correspondence with $\supp(c)=Z$ and $\gamma\in k[U]$.
Then $c$ is equivalent to the correspondence $(U\times\A^1,(\phi,x-\gamma),f\circ pr_U)$ where $x$ denotes the coordinate on the additional copy of $\A^1.$
\end{lem}
\begin{proof}
We may assume that $U$ is affine. Consider the data $d=(U\times\A^1\times\A^1,(\phi,x-t\gamma),f\circ pr_U).$ Then $k[U][x,t]/(\phi,x-t\gamma)\cong k[U]/\phi[t]$ is finite over $\Spec k[t]$ and the map $k[\A^n][x,t]\to k[U][x,t]/(\phi,x-t\gamma)$ is surjective, so $\supp(d)$ is a closed subset of $\A^n\times\A^1\times\A^1$ which is finite over $\A^1=\Spec k[t].$ Then $d$ defines an equivalence between $(U,\phi,f)\sim (U\times\A^1,(\phi,x),f\circ pr_U)$ and $(U\times\A^1,(\phi,x-\gamma),f\circ pr_U).$
\end{proof}

\begin{lem}\label{alphachange}
Let $A\in \gSL_n(k)$ be a matrix with trivial determinant over $k$. Suppose that $c=(U,\phi,f)$ represents a correspondence in $Fr_n(X,Y)$. Let $U_A$ be an \'etale neighbourhood of the closed subset $A\cdot Z$ in $\A^n$ given by $U_A=U\to\A^n\stackrel{A}\to\A^n$. Then $\phi$ is a framing of $A\cdot Z$ in $U_A.$
We claim that $c$ is equivalent to  $(U_A,\phi,f)$ in $H_0(\ZZ F(\Delta^{\bullet}_k,Y))$.
\end{lem}
\begin{proof}
Any matrix with trivial determinant can be derived from the identity matrix by a sequence of elementary homotopies $\A^1\to \gSL_n,$ so we may assume that there is a homotopy $H\colon\A^1\to \gSL_n$ with $H(0)=Id$ and $H(1)=A$. Then $H$ is an automorphism of $\A^n_{\A^1}$. Let $(\A^1\times U)_H$ denote an \'etale neighborhood of $H\cdot(\A^1\times Z)$ given by $(\A^1\times U)_H=\A^1\times U\to\A^1\times\A^n\stackrel{H}\to\A^1\times\A^n.$ Then the data $d=((\A^1\times U)_H,\phi\circ pr_U,f\circ pr_U)$ define a correspondence in $Fr_n(\A^1,Y)$.
Since $i_0^*(d)=c$ and $i_1^*(d)=c'$ the lemma follows.
\end{proof}

\begin{cor}\label{swap}
Let $L$ be a a finite field extension of $k$ and $\mu_1,\mu_2\in L, \lambda\in L^{\times}$ such that $L=k(\mu_1,\mu_2)$.
Let $f\in Y(L).$
Then the data $c=(\A^2_L,(\lambda(x_1-\mu_1),x_2-\mu_2),f)$ and $c'=(\A^2_L,(x_1+\mu_2,\lambda(x_2-\mu_1)),f)$ define framed correspondences in $Fr_2(\Spec k,Y)$ which are equivalent in $H_0(\ZZ F(\Delta^{\bullet}_k,Y))$
\end{cor}
\begin{proof}
Take $\bigl(\begin{smallmatrix}
0&-1\\ 1&0
\end{smallmatrix} \bigr)\in \gSL_2(k)$. Then by~\ref{alphachange} $(\A^2_L,(\lambda(x_1-\mu_1),x_2-\mu_2),f)\sim ((\A^2_L)_A,(\lambda(x_1-\mu_1),x_2-\mu_2),f)=(\A^2_L,(\lambda(x_2-\mu_1),-x_1-\mu_2),f)\stackrel{\ref{SL}}\sim(\A^2_L,(x_1+\mu_2,\lambda(x_2-\mu_1),f).$
\end{proof}

\begin{dfn}
Let $c=(\alpha,U,Z,\phi_1,\ldots\phi_n)$ be a level $n$ correspondence. Then the closed subscheme $U\times_{\phi}0$ is finite over $k$ and we define the degree of the correspondence $c$ defined by $\deg c$ as the degree of this subscheme over $k$, i.e.
\[\deg c=\dim_k k[U\times_{\phi}0.]\]
\end{dfn}
\begin{rem}\label{degree}
The degree is additive: $\deg(c+c')=\deg(c)+\deg(c')$ and is preserved by the equivalence relation: $\deg(c)=\deg(c')$ if $c\sim c'.$
\end{rem}

\section{Multiplicative structure and transfer map on $H_0(\ZZ F(\Delta^{\bullet}_k,\Gm^{\wedge *}))$.}\label{MultStruct}
\subsection*{Multiplicative structure}
For $X,X_1,Y,Y_1\in\Sm_k$ there is an obvious product structure
\[Fr_n(X,Y)\times Fr_m(X_1,Y_1)\to Fr_{n+m}(X\times_k X_1,Y\times_k Y_1)\]
defined as follows: if $a=(U,\phi,f)$ represents a correspondence in $Fr_n(X,Y)$ and $b=(U',\phi',f')$ in $Fr_n(X_1,Y_1)$ then $a\times b$ is represented by data $a\times b=(U\times_kU',\phi\times\phi',f\times f')$ in $Fr_{n+m}(X\times X_1,Y\times Y_1).$
Applying an isomorphism $Fr_n(X\times\A^1\times X_1,Y\times Y_1)\cong Fr_n(\A^1\times X\times X_1,Y\times Y_1)$ we see that it descends to the product
$H_0(\ZZ F_n(\Delta^{\bullet}_k\times X,Y))\times H_0(\ZZ F_m(\Delta^{\bullet}_k\times X_1,Y_1))\to H_0(\ZZ F_{n+m}(\Delta^{\bullet}_k\times X\times X_1,Y\times Y_1))$.
Moreover, this product is compatible with composition:
\begin{rem}\label{prodcomp}
For $f_i\in H_0(\ZZ F_{n_i}(\Delta^{\bullet}_k\times X_i,Y_i))$ and $g_i\in H_0(\ZZ F_{m_i}(\Delta^{\bullet}_k\times Y_i,Z_i))$ for $i=1,2$
we have that $(g_1\times g_2)\circ(f_1\times f_2)=(g_1\circ f_1) \times(g_2\circ f_2)$ in $H_0(\ZZ F_{n}(\Delta^{\bullet}_k\times X_1\times X_2,Z_1\times Z_2))$ where $n=n_1+n_2+m_1+m_2.$
\end{rem}

\begin{lem}
Let $(U,\phi,f)$ represent a correspondence in $Fr_n(X,Y)$. Then $(U\times\A^1,\phi\times id,f\circ pr_U)$ and $(\A^1\times U,id\times\phi,f\circ pr_U)$ are equivalent in $H_0(\ZZ F(\Delta^{\bullet},Y)(k))$
\end{lem}
\begin{proof}
By the lemmas~\ref{SL} and~\ref{alphachange} we have $(U\times\A^1,\phi\times id_{\A^1},f\circ pr_U)=(U\times\A^1,-id_{\A^1}\times\phi,f\circ pr_U)=(U',-id_{\A^1}\times\phi,f\circ pr_U)$ Here the neighbourhood $U'$ is given by \[U'=U\times\A^1\to\A^n_X\times\A^1\stackrel{(v,x,t)\mapsto(-t,v,x)}\to\A^n\times\A^1_X\text{ for }v\in\A^n,t\in\A^1,x\in X.\]
Note that the latter correspondence is equivalent to $(\A^1\times U,id_{\A^1}\times\phi,f\circ pr_U).$
\end{proof}

This lemma shows that multiplication is compatible with the stabilization map $\sigma_Y$, so it descends to
 \[H_0(\ZZ F(\Delta^{\bullet}_k\times X,Y))\times H_0(\ZZ F(\Delta^{\bullet}_k\times X_1,Y_1))\to H_0(\ZZ F(\Delta^{\bullet}_k\times X\times X_1,Y\times Y_1)).\]

This construction for $X=\Spec k$ and $Y=\Gm^{*}$ gives a multiplicative structure on the graded abelian group $H_0(\ZZ F(\Delta^{\bullet}_k,\Gm^{*})$ which descends to the multiplication on the graded abelian group $H_0(\ZZ F(\Delta^{\bullet}_k,\Gm^{\wedge*}).$

\subsection{Transfer map on $H_0(\ZZ F(\Delta^{\bullet}_k,\Gm^{\wedge*}))$}

Let $L$ be a finite extension of $k$. We will construct the map $tr_{L/k}\colon H_0(\ZZ F(\Delta^{\bullet}_L,\Gm^{\wedge*}))\to H_0(\ZZ F(\Delta^{\bullet}_k,\Gm^{\wedge*}))$ as follows:

Suppose that $v_1,\ldots v_n\in L$ is an $k$-basis of $L$ and $v_1\in k$.
\begin{dfn}
Define $c_{L/k}$ as a framed correspondence represented by the data
$c=(\A^n_L,(x_1-v_1,\ldots,x_n-v_n),pr_L)$. Since $v_1,\ldots,v_n$ generate $L$ over $k$, we get that $supp(c)=Z\to\A^n_L\to\A^n_k$ is a closed embedding, so $c$ defines a correspondence in $Fr_n(\Spec k,\Spec L)$.
\end{dfn}
\begin{lem}\label{ccorrect}
The class of element $c_{L/k}$ in $H_0(\ZZ F(\Delta^{\bullet}_k,\Spec L)$ does not depend on the choice of a basis set.
\end{lem}
\begin{proof}
Suppose that $u_1,\ldots, u_n$ with $u_1\in k$ is another basis set. Then there is a matrix $A\in \gSL_n(k)$ such that
$(v_1,\ldots,v_n)\cdot A=(\lambda u_1,u_2,\ldots,u_n)$ for some $\lambda\in k^{\times}.$

Denote by $x=(x_1,\ldots,x_n)$ the vector of coordinates on $\A^n_L$ and $v=(v_1,\ldots,v_n)$ the vector in $L^n.$
Then by lemma~\ref{alphachange} we have $c_{L/k}=(\A^n_L,x-v,pr_L)\sim((\A^n_L)_A,x-v,pr_L)=(\A^n_L,A^{-1}\cdot x-v,pr_L)\sim(\A^n_L,x-A\cdot v,pr_L)$ by lemma~\ref{SL}.
Since $u_1\in k$, Lemma~\ref{addone} implies $(\A^n_L,x-A\cdot v,pr_L)=(\A^n_L,(x_1-\lambda u_1,x_2-u_2,\ldots, x_n-u_n),pr_L)\sim (\A^{n-1}_L,(x_2-u_2,\ldots, x_n-u_n),pr_L)\sim(\A^n_L,(x_1-u_1,x_2-u_2,\ldots, x_n-u_n),pr_L).$
\end{proof}

\begin{rem}\label{generator}
Chose a generator $\alpha\in L,$ so $L=k(\alpha)$. Then $c_{L/k}$ is equivalent to the correspondence $(\A^1_L,x-\alpha,pr_L)$ in $H_0(\ZZ F(\Delta^{\bullet}_k,\Spec L)$
\end{rem}
\begin{proof}
Take a basis $1,\alpha,\alpha^2,\ldots,\alpha^{n-1}.$ Thus we may define $c_{L/k}$ as $c_{L/k}=(\A^n_L,x_1-1,x_2-\alpha,\ldots,x_n-\alpha^{n-1},pr_L)$. Consider the data $d=(\A^n_L\times_L\A^1_L,x_1- t,x_2-\alpha,x_3-\alpha^2 t,\ldots,x_n-\alpha^{n-1} t,pr_L)$. Then $\supp(d)=\Spec L[x_1,\ldots,x_n, t]/(x_1- t,x_2-\alpha,x_3-\alpha^2 t,\ldots,x_n-\alpha^{n-1} t,pr_L)=\Spec L[ t]$ is finite over $\A^1_k=\Spec k[ t]$, and since $k(\alpha)=L,$ we have that $\supp(d)\to A^{n+1}_L\to\A^{n+1}_k$ is a closed embedding. Thus $c_{L/k}$ is equivalent to the correspondence $(\A^1_L,x-\alpha,pr_L)$ and does not depend on a choice of generator $\alpha.$
\end{proof}

\begin{dfn}
$tr_{L/k}=(c_{L/k})^{*}\colon H_0(\ZZ F(\Delta^{\bullet}_L,\Gm^{\wedge*})\to H_0(\ZZ F(\Delta^{\bullet}_k,\Gm^{\wedge*})$ is a map induced by the composition with $c_{L/k}\in H_0(\ZZ F(\Delta^{\bullet}_k,\Spec L)).$
\end{dfn}

\begin{lem}\label{transitive}
Suppose $k\subseteq L'\subseteq L$. Then $tr_{L/k}=tr_{L'/k}\circ tr_{L/L'}.$
\end{lem}
\begin{proof}
Take $\alpha,\beta\in L$ such that $L'=k(\alpha)$ and $L=L'(\beta)$. Then there are $\lambda,\mu\in k$ such that $k(\lambda\alpha+\mu\beta)=L.$ Then by Remark~\ref{generator} it is sufficient to check that $(\A^1_L,x-\beta,pr_{L})\circ(\A^1_{L'},x-\alpha,pr_{L'})= (\A^1_L,x-(\lambda\alpha+\mu\beta),pr_L)$ in $H_0(\ZZ F(\Delta^{\bullet}_k,\Spec L)).$ The left hand side is given by $(\A^2_L,(x_1-\alpha,x_2-\beta),pr_L)$. There is a matrix $A\in SL_2(k)$ such that $(\alpha,\beta)\cdot A=(\lambda\alpha+\mu\beta,\gamma)$ for some $\gamma\in L^{\times}$. Then by~\ref{alphachange} we have $(\A^2_L,(x_1-\alpha,x_2-\beta),pr_L)=((\A^2_L)_A,(x_1-\alpha,x_2-\beta),pr_L)\sim(\A^2_L,(x_1-(\lambda\alpha+\mu\beta),x_2-\gamma),pr_L)\stackrel{\ref{addone}}=(\A^1_L,x-(\lambda\alpha+\mu\beta),pr_L).$
\end{proof}

\subsubsection{Projection formula}
Note that the pullback homomorphism (composition with the projection $p:\Spec L\to\Spec k$) $p^*\colon H_0(\ZZ F(\Delta^{\bullet}_k,\Gm^{\wedge *}))\to H_0(\ZZ F(\Delta^{\bullet}_L,\Gm^{\wedge *}))$  is a ring homomorphism. We prove that the transfer map $tr_{L/k}\colon H_0(\ZZ F(\Delta^{\bullet}_L,\Gm^{\wedge *}))\to H_0(\ZZ F(\Delta^{\bullet}_k,\Gm^{\wedge *}))$ is an $ H_0(\ZZ F(\Delta^{\bullet}_k,\Gm^{\wedge *}))$-module homomorphism:
\begin{lem}\label{projection}
$(p^*x\cdot y)\circ c_{L/k}=x\cdot (y\circ c_{L/k})$ and $(y\cdot p^*x)\circ c_{L/k}=(y\circ c_{L/k})\cdot x$ in for any $x\in H_0(\ZZ F(\Delta^{\bullet}_k,\Gm^{\wedge *}))$ and $y\in H_0(\ZZ F(\Delta^{\bullet}_L,\Gm^{\wedge *})).$
\end{lem}
\begin{proof}
Let $x$ be represented by data $(U,\phi,f)\in Fr_*(k,\Gm^{\wedge a})$ and $y$ by $(V,\gamma,g)\in Fr_*(L,\Gm^{\wedge b})$. Then $p^*x=(U_L,\phi,f)$ and $p^*x\cdot y=(U_L\times_L V,(\phi,\gamma),(f,g))$ and
\[
(p^*x\cdot y)\circ c_{L/k}=(\A^1_L\times_L U_L\times_L V,(x-\alpha,\phi,\gamma),(f,g))
\]
where $t$ is the coordinate on $\A^1_L$ and $\alpha$ generates $L$ over $k$ and
By lemma~\ref{alphachange} and ~\ref{SL} we have that
$(\A^1_L\times_L U_L\times_L V,(x-\alpha,\phi,\gamma),(f,g))\sim
(\A^1_L\times_L U_L\times_L V,(\phi,-(x-\alpha),\gamma),(f,g))\sim(U_L\times_L\A^1_L\times_L V,(\phi,x-\alpha,\gamma),(f,g))$
Note that $U_L\times_L\A^1_L\times_L V=U\times_k(\A^1_L\times_L V)$ so we get that $(p^*x\cdot y)\circ c_{L/k}= x\cdot(y\circ c_{L/k})$.
\end{proof}

\section{Moving lemma}\label{MovingSection}

The aim of this section is to prove the moving lemma~\ref{Moving}.

\begin{lem}\label{hsurface}
Let $x\in\PP^N_k$ be a closed point in a projective space. Then there exists an integer $d_x$ such that
for any finite subset of closed points $Y\subset\PP^N_k$, $x\notin Y$, there is a hypersurface $H$ defined by a single homogeneous polynomial $F$ of degree $d_x$ such that $x\in H$ and $H$ is disjoint from $Y$.
\end{lem}
\begin{proof}
First, we may assume that $x$ and $Y$ are contained in some affine space $\A^N_k\subset\PP^N_k$ with coordinates $t_i.$
Let $x=(x_1,\ldots,x_N)$ for some $x_i\in k(x)$. Let $p_i$ be a minimal polynomial for $x_i$ over $k$. For any $\lambda_i\in k$ consider $F_0=\sum_{i=1}^N\lambda_ip_i(t_i)$. Let $H_0$ denote the set of zeroes of $F$. Then $x\in H_0$ for any choice of $\lambda_i$. Now we check that one can choose $\lambda_i$ such that $H_0$ is disjoint from $Y.$ Note that for any $y\in Y$ the vector $v_y=(p_1(y_1),\ldots,p_N(y_N))$ is non-zero in $\A^N(k(y)).$ Then for $\lambda\in\A^N(k)$ the condition $\lambda\cdot v_y=0$ defines a proper linear subspace in $\A^N(k)$. Since $k$ is infinite, then there is $\lambda\in\A^N(k)$ such that $\lambda\cdot v_y\neq 0$ for all $y\in Y$. With this choice of $\lambda$ we have that $H_0$ is disjoint from $Y$, so its closure $H$ defined by the homogenization $F$ of $F_0$ is the desired hypersurface.
\end{proof}

\begin{lem}\label{Bertini}
Let $\overline{X}$ be a projective curve, $Z$ be its closed point. Assume that $\overline{X}\setminus Z$ is smooth and $X_{\infty}$ be a finite set of closed points of $\overline{X}\setminus Z$. Let $D$ be an effective divisor and suppose that $|D|$ has no base point. Then there is a simple divisor $D'$ in $|D|$ such that $\supp(D')$ is disjoint from $Z$ and $X_{\infty}.$
\end{lem}
\begin{proof}
Let $D$ be an effective divisor and suppose that $|D|$ has no base point. Then let $D'$ be a divisor from $|D|$ such that $\supp(D')$ is disjoint with $\supp(D)$. Then there are two sections $s_1,s_2$ of $L(D)=L(D')$ such that $\divv(s_1)=D$ and $\divv(s_2)=D'.$ Then they define a regular function $f=[s_1\colon s_2]\colon \overline{X}\to\PP^1$. The fibers over all points except $f(Z)$ are contained in the smooth curve $\overline{X}\setminus Z$. Since $\charr k=0$, \cite[III, Corollary 10.7]{Hartshorne} implies that $f\colon \overline{X}\setminus Z\to\PP^1$ is smooth over some open subset $U$ of $\PP^1.$ Pick a rational point $x\in U\setminus f(Z)\setminus f(X_{\infty}).$ Then the fiber over $x$ defines a divisor $D'$  with the desired properties.
\end{proof}

\begin{lem}\label{finitefunction}
Suppose $\overline{X}$ is a projective curve over $k$, $Z$ is its closed point and $\overline{X}\setminus Z$ is smooth over $k$.
Let $X_{\infty}$ be closed subset disjoint with $Z$.
Let $\phi\in k(\overline{X})$ be a rational function that is defined in a neighbourhood of $Z$ and $\phi(Z)=0.$
Then there is a rational function $\psi\in k(\overline{X})$ such that
\begin{itemize}
\item [(1)] $\psi/\phi$ is defined and invertible in a neighborhood of $Z$
\item [(2)] $\psi$ has a pole at any $x\in X_{\infty},$
\item [(3)] $\psi$ considered as a regular function $\psi\colon\overline{X}\setminus Z\to\PP^1$ is smooth over $0.$
\end{itemize}
\end{lem}
\begin{proof}
Let $D$ be a Cartier divisor defined by a system $((U,\phi),(\overline{X}\setminus Z,1))$ where $U$ is an open neighborhood of $Z$ such that $\phi$ is defined and invertible on $U\setminus Z.$ Choose an embedding $\overline{X}\to\PP^N$ for some $N$ and the corresponding very ample bundle $D_1$. By Serre's theorem~\cite[II, Theorem 5.17]{Hartshorne} there is $n$ such that the line bundle $L(D+nD_1)$ is generated by global sections. By Lemma~\ref{Bertini}, there is a simple divisor $D'\sim nD_1$ away from $Z$ and $X_{\infty}$.

Since $L(D+D')$ is generated by global sections, the linear system $|D+D'|$ is base-point free~\cite[II, Lemma 7.8]{Hartshorne}. Then by Lemma~\ref{Bertini} $D+D'\sim M$ where $M$ is a simple divisor supported away from $Z\cup X_{\infty}\cup\supp(D')$.

Write $M=y_1+\ldots y_k$ where $y_i$ are points on $\overline{X}\setminus Z.$ By Lemma~\ref{hsurface} applied to the embedding $\overline{X}\to\PP^N$ there is a very ample divisor $H_i$ containing $y_i$ with $\supp(H_i)$ disjoint from $\supp(D+D')\cup X_{\infty}$ Then $M'=M-\sum_i=1^k H_i$ is an effective divisor. For each $i$ we can choose $H'_i$ by lemma~\ref{hsurface} such that $\supp(H'_i)$ is disjoint from $\supp(\sum H_i)$. Then $M'\sim M-\sum H'_i$, so the linear system $|M'|$ has no base points. Thus by Lemma~\ref{Bertini} we may assume that $M'$ is simple and disjoint from $Z\cup X_{\infty}$. Note that $M+M'\sim\sum H_i$ is very ample. Now fix the embedding $\overline{X}\to\PP^{N_1}$ so that $D+D'+M'\sim M+M'$ becomes a hyperplane section for this embedding.

For any $x\in X_{\infty}$ choose by lemma~\ref{hsurface} applied to the embedding the $\overline{X}\to\PP^{N_1}$ the divisor $H_x$
such that $H_x$ contains $x$ and $\supp(H_x)$ is disjoint from $\supp(D+D'+M').$ Then $H_x\sim d_xL$ where $L$ is any hyperplane section, so $H_x\sim d_x(M+M').$ Let $d=\sum_{x\in X_{\infty}} d_x$. Then $d(M+M')\sim\sum_{x\in X_{\infty}} H_x$.

Then $\sum_{x\in X_{\infty}} H_x=d(M+M')\sim (D+D'+M')+(d-1)(M+M').$ By Lemma~\ref{Bertini} $(d-1)(M+M')$ is equivalent to a simple divisor $M''$ supported away from $Z\cup X_{\infty}\cup\supp(D')\cup\supp(M')\cup\supp(\sum_{x\in X_{\infty}}H_x))$. Thus $M'''=D'+M'+M''$ is a simple divisor supported away from $Z$ and $X_{\infty}$.

So, finally we get two equivalent divisors $\sum_{x\in X_{\infty}} H_x\sim D+M'''$ on $\overline{X}$. Let $L$ be the corresponding line bundle. Then we have a global section $s_1$ such that $\divv(s_1)=D+M'''$ and $s_2$ such that $\divv(s_2)=\sum_{x\in X_{\infty}} H_x$.
By construction, $\divv(s_1)$ and $\divv(s_2)$ are disjoint, then there is a regular map $\psi=[s_1:s_2]\colon\overline{X}\to\PP^1.$ Since $\divv_0(\psi)=D+M'''$, we have that $\psi$ defines the same Cartier divisor as $\phi$ in a neighborhood of $Z$, so $\psi$ satisfies the condition $(1).$ Since $X_{\infty}\subseteq \divv_{\infty}(\psi)$, then $\psi$ satisfies the condition $(2).$ Moreover, since $M'''$ is a simple divisor on $\overline{X}\setminus Z$, $\psi$ satisfies the condition $(3).$
\end{proof}

\begin{lem}\label{psideformation}
Consider a framed correspondence $c=(U,\phi=(\phi_1,\ldots\phi_n),f)\in Fr_n(k,Y)$ such that $Z=\supp(c)$ is a single point and $\alpha\colon U\to\A^n$ is an \'etale neighborhood of $Z$.
Let $X$ denote the curve defined in $U$ by the equations $\phi_2=0,\ldots\phi_n=0$. Suppose there is a finite map $\psi\colon X\to\A^1$ with $\psi^{-1}(0)=Z_0=Z\coprod Z'$ and $\psi^{-1}(0)=Z_1.$
Then
\begin{itemize}
\item[(1)] There is an integer $m$ and a morphism $\rho\colon U\to\A^m$ such that $\rho(Z)=0$ and the map $U\stackrel{\alpha\times\rho}\to\A^n\times\A^m$ is injective.
\item[(2)] The correspondence $c=(U\setminus Z',(\psi,\phi_2,\ldots,\phi_n),f)$ is equivalent to the difference $c'-c''$ where
\[c'=(U\times\A^m,(\psi-1,(\phi_i(u))_{i=2}^n,(t_i-\rho_i(u))^{m}_{i=1}),f\circ pr_U)\]
\[c''=((U\setminus Z)\times\A^m,(\psi,(\phi_i(u))_{i=2}^n,(t_i-\rho_i(u))^{m}_{i=1}),f\circ pr_U).\]
\end{itemize}
\end{lem}
\begin{proof}
Let $L$ denote the residue field of the point $Z$. We may assume that $U$ is affine. Let $\rho'\colon U\to\A^m$ be some closed embedding. Then $\rho'(Z)$ is a point in $\A^m$ with residue field $L$. Since $\alpha$ gives a closed embedding of $Z$ in $\A^n$ we have that $\alpha(Z)$ is also a point with residue field $L$. Then there is some regular map $g\colon\A^n\to\A^m$ with $g(\alpha(Z))=\rho'(Z)$. Take $\rho=\rho'-g\circ i\colon U\to\A^m.$ Let is check that $\alpha\times\rho\colon\A^n\times\A^m$ is injective on all points: for any $u\in U$ $(\alpha(u),\rho(u))=(\alpha(u'),\rho(u'))$ implies $g(\alpha(u))=g(\alpha(u')$, then $\rho'(u)=\rho'(u')$ so $u=u'.$

Let $t_1,\ldots,t_m$ denote the coordinates on $\A^m$.
Consider a data $d=(U\times\A^m\times\A^1,(\psi(u)-t,(\phi_i(u))_{i=2}^n,(t_i-\rho_i(u))_{i=1}^m),f\circ pr_U)$ where $t$ is the coordinate on the additional copy of $A^1.$ The framing in $d$ defines a curve $X$ embedded into $U\times\A^m\times\A^1$ via the composition
\[X\stackrel{\Gamma_{\psi}}\to X\times\A^1 \to U\times \A^1\stackrel{\Gamma_{\rho}\times id}\to U\times\A^m\times\A^1\]
Since $X$ is finite over $\A^1$ and $\alpha\times\rho$ is injective, we get a closed embedding $X\to\A^n\times A^m\times\A^1$. Thus the data $d$ defines a framed correspondence in $Fr_{n+m}(\A^1,Y)$ from $A^1$. Its fiber over zero is a framed correspondence
\[d_0=(U\times\A^m,(\psi(u),(\phi_i(u))_{i=2}^n,(t_i-\rho_i(u))_{i=1}^m),f\circ pr_U).\]

In $\ZZ F_{n+m}(\Spec k,Y)$ the correspondence $d_0$ is equivalent to the sum $d_Z+d_{Z'}$ where
\[d_Z=((U\setminus Z')\times\A^m,(\psi(u),(\phi_i(u))_{i=2}^n,(t_i-\rho_i(u))_{i=1}^m),f\circ pr_U)\]
\[d_{Z'}=((U\setminus Z)\times\A^m,(\psi(u),(\phi_i(u))_{i=2}^n,(t_i-\rho_i(u))_{i=1}^m),f\circ pr_U)\]
Note that the radical of the ideal $(\psi,\phi_2,\ldots,\phi_n)$ equals to the maximal ideal defining the point $Z,$ so it contains $\rho_i(u)$ for each $i$. Then lemma~\ref{nilpotent} implies that $d_Z$ is equivalent in $H_0(\ZZ F(\Delta^{\bullet}_k,Y))$ to
\[
((U\setminus Z')\times\A^m,(\psi(u),(\phi_i(u))_{i=2}^n,(t_i)_{i=1}^m),f\circ pr_U)\sim (\alpha, (U\setminus Z'),(\psi,(\phi_i)_{i=2}^n),f\circ pr_U).
\]
The lemma follows.
\end{proof}

\begin{dfn}
Consider a correspondence $c=(U,\phi,f)\in Fr_n(k,Y)$. We will call it simple, if the scheme $U\times_{\phi}0$ is smooth over $k.$
\end{dfn}

\begin{lem}\label{simple_reduction}
Consider a level $n$ framed correspondence $c=(U,\phi,f)\in Fr_n(k,Y)$ with $Z$ a single point.
Then $c$ is equivalent in $H_0(\ZZ F(\Delta^{\bullet}_k,Y))$ to a difference of two simple correspondences.
\end{lem}
\begin{proof}
Consider a composition $\theta\colon U\setminus Z\stackrel{\phi}\to\A^n\setminus 0\to\PP^{n-1}$. Since $\charr k=0$ $\theta$ is smooth over some open subset $V$ of $\PP^{n-1}$(\cite[III, Corollary 10.7]{Hartshorne})
Take a rational point $a\in V$. It defines a line $l$ in $\A^n$ and $\theta^{-1}(a)=\phi^{-1}(l)\setminus Z$ is smooth. Since $a$ is rational, there is a matrix $A\in\gSL_n(k)$ that moves the coordinate line $(*,0,\ldots,0)$ to $l$, so $(\phi\cdot A)^{-1}(*,0,\ldots,0)=\phi^{-1}(l)$. Then by lemma~\ref{SL} we may assume that
$l$ coincides with the coordinate line $(*,0,\ldots,0)$, so $X\setminus Z$ is smooth where the curve $X$ is defined in $U$ by the equations $\phi_2=0,\ldots\phi_n=0.$

There is a projective curve $\overline{X}$ with an open embedding $X\to\overline{X}$ and $\overline{X}\setminus Z$ smooth. Take $X_{\infty}=\overline{X}\setminus X$. By lemma~\ref{finitefunction} there is a regular map $\psi\colon\overline{X}\to\PP^1$ with the properties $(1)-(3)$. It is finite since $\overline{X}$ is projective and irreducible. Take $\A^1=\PP^1\setminus{\infty}$ and $X_0=\psi^{-1}(\A^1)$. Then $X_0$ is an open subset of $X$ and $Z\subset X_0.$ Take $U_0$ be an open subset of $U$ such that $X_0=X\cap U_0.$ Note that $\psi\colon X_0\to\A^1$ is finite and generically smooth. Then it is smooth over some rational point of $\A^1$. Without loss of generality we may assume that $\psi$ is smooth over $1$.

We have $\psi^{-1}(0)=Z\sqcup Z'$. Then the correspondence $\tilde{c}=(U\setminus Z',(\psi,\phi_2,\ldots\phi_n),f)$ is equivalent to a difference $c'-c''$ given by lemma~\ref{psideformation}. Since $\overline{X}\setminus Z$ is a smooth curve and $\psi\colon\overline{X}\setminus Z\to\PP^1$ is smooth over $1$ and $0$, the correspondences $c'$ and $c'$ are simple.

Recall that $\psi/\phi_1$ is invertible on $U\setminus Z'.$ By Lemma~\ref{retract} we may assume that there is a projection $U\to\Spec k(Z).$ Let $\lambda=\psi/\phi_1(Z)\in k(Z)^{\times}.$  Using the latter projection we can consider $\lambda$ as a constant regular function on $U.$ Moreover, $\psi=\gamma\phi_1=(\lambda+\delta)\phi_1$ where $\delta$ lies in the maximal ideal $M_Z.$ Then by lemma~\ref{multiplication} $\tilde{c}$ is equivalent to $(U\setminus Z',(\lambda\phi_1,\phi_2,\ldots\phi_n),f).$ Note that the choice of $\psi$ depends only on the curve $X,$ hence on the ideal $(\phi_2,\ldots,\phi_n).$ Then by the same reasoning the correspondence $b=(U\setminus Z',(\lambda\phi_1,\lambda^{-1}\phi_2,\ldots\phi_n),f)$ is equivalent to a difference of two simple correspondences. By lemma~\ref{SL}, $b\sim c$, hence the lemma.
\end{proof}

\begin{lem}\label{minuslambda}
Let $c=(U,\phi,f)$ represent a correspondence in $Fr_{n-1}(\Spec k,Y).$ Assume that there is a projection $U\to Z$ and $\lambda\in k(Z).$ Let $t$ denote the coordinate on $\A^1.$
Then $c$ is equivalent to the correspondence $(\A^1\times U,(t-\lambda,\phi),f\circ pr_U)$ in $H_0(\ZZ F(\Delta^{\bullet}_k,Y))$.
\end{lem}
\begin{proof}
Consider an affine line $\A^1$ with coordinate $\theta$ and a data
\[d=(\A^1\times\A^1\times U,(t-\theta\lambda,\phi),f\circ pr_U).\]
We may assume that $U$ is affine. Then $k[U][t,\theta]/(t-\theta\lambda,\phi)=(k[U']/(\phi))[t,\theta]/(t-\theta\lambda)=k[U]/(\phi)[\theta]$, so $\supp(d)=\A^1\times Z$ is finite over $\A^1$, where $Z=\supp(c).$ Then $d$ defines a deformation in $Fr_n(\A^1,Y)$ which establishes the desired equivalence.
\end{proof}

\begin{lem}\label{anyf}
Let $Y$ be an open subset of an affine space $\A^m$.
Let $c=(U,\phi,f)$ represent a correspondence in $Fr_n(\Spec k,Y)$. Denote $Z=\supp(c)$ and let $f'\colon U\to Y$ be a regular function such that two restrictions coincide: $f'|_Z=f|_Z$. Then $c$ is equivalent to the correspondence $c'=(U,\phi,f')$ in $H_0(\ZZ F(\Delta^{\bullet}_k,Y))$.
\end{lem}
\begin{proof}
Let $\A^1=\Spec k[t]$. Consider the function $F=tf+(1-t)f'\colon U\times\A^1\to \A^m$. Denote $W=F^{-1}(Y).$ Then $W$ contains $Z\times\A^1$. Consider the data $d=(W,\phi,tf+(1-t)f')$ Then $\supp(d)=Z\times\A^1$ and $d$ defines a correspondence in $Fr_n(\A^1,Y)$ which deforms $c$ to $c'$.
\end{proof}

\begin{lem}\label{minusone}
Let $Y$ be an open subset of an affine space $\A^m$. Suppose $Z$ is a closed point in $\A^n$ and $\alpha=(\alpha_1,\ldots,\alpha_n)\colon U\to\A^n$ is its \'etale neighborhood. Assume that there is a projection $U\to Z.$ Let $\lambda=\alpha_1(Z)\in k(Z).$ Suppose we have a simple correspondence $c=(U,(\alpha_1-\lambda,\phi),f)\in Fr_n(k,Y)$ for some $\phi=(\phi_2,\ldots,\phi_n)$ with $\supp(c)=Z.$

Then $c$ is equivalent to some simple correspondence $c'\in Fr_{n-1}(k,Y).$
\end{lem}
\begin{proof}
The element $\lambda$ defines the closed embedding $\Spec k(\lambda)\to\A^1.$ Define a closed subset $U'\subseteq U$ as the pullback of the diagram
\[\Spec k(\lambda)\times\A^{n-1}\to\A^1\times\A^{n-1}\stackrel{\alpha}\leftarrow U.\]
Denote by $\alpha'$ the composition $\alpha'\colon U'\to\Spec k(\lambda)\times\A^{n-1}\to\A^{n-1}.$ Then $id\times\alpha'\colon\A^1\times U'\to\A^n$ is an \'etale neighborhood of $Z$ in $\A^n.$ Take $U''=U\times_{\A^n} (\A^1\times U')$. Then the correspondence $c$ is represented by explicit correspondence $(U'',((\alpha_1-\lambda)\circ pr_U,\phi\circ pr_U),f\circ pr_U).$ Note that $U'$ is a closed subset of $U''$ and there is a retraction $p\colon U''\to U'$

The regular function $\phi_i\circ pr_U-\phi_i\circ p$ vanishes on $U'$, hence $\phi_i\circ p=\phi_i\circ pr_U+\delta_i$ for some $\delta_i\in(\alpha_1-\lambda)\circ pr_U$. Then by lemma~ \ref{nilpotent} we have that $c$ is equivalent to the correspondence $\tilde{c}=(U'',((\alpha_1-\lambda)\circ pr_U,\phi\circ p),f).$ By Lemma~\ref{anyf} we have that $\tilde{c}$ is equivalent to $(U'',((\alpha_1-\lambda)\circ pr_U,\phi\circ p),f(Z))$ where $f(Z)$ is a constant function on $U''$.

By construction, there is an \'etale map $(\alpha_1\circ pr_U,p)\colon U''\to\A^1\times U'$
It gives an equivalence between $\tilde{c}$ and a correspondence $\A^1\times U',(t-\lambda,\phi|_{U'}),f(Z))$ which is equivalent to the level $n-1$ correspondence $c'=(U'\phi|_{U'},f(Z))$ by~\ref{minuslambda}. Note that the scheme $U'\times_{\phi}0$ equals to $U\times_{(\alpha_1-\lambda,\phi)}0$ is smooth, so $c'$ is simple.
\end{proof}

\begin{lem}\label{one}
Let $Y$ be an open subset of an affine space $\A^m$. Let $c=(U,\phi,f)\in Fr_n(k,Y)$ be a simple correspondence. Then $c$ is equivalent to a sum of level $1$ simple correspondences.
\end{lem}
\begin{proof}
Let $n>1.$
We may assume that $Z$ is a single point. Since $U\times_{\phi}0$ is smooth over $k$, $\phi$ is a regular system, so the residues $\overline{\phi_1},\ldots,\overline{\phi_n}$ form a basis of the vector space of $\Ms_Z/\Ms_Z^2$ over $k(Z)=\Os_{U,Z}/\Ms_Z.$ Then there is a matrix $A\in\gSL_n(k(Z))$ such that $A\cdot(\overline{\phi_1},\ldots,\overline{\phi_n})=(\overline{\alpha_1-\lambda},\phi'_2,\ldots,\phi'_n$ where $\lambda=\alpha_1(Z)\in k(Z).$ Applying of $A$ does not change the equivalence class of $c$, by~\ref{SL} hence we may assume that $\overline{\phi_1}=\overline{\alpha_1-\lambda}$, so $\phi_1=(\alpha_1-\lambda)(1+m)+\delta$ with $m\in \Ms_Z$ and $\delta\in(\phi_2,\ldots\phi_n).$ Then by~\ref{multiplication} and~\ref{nilpotent} we may assume that $\phi_1=\alpha_1-\lambda.$ Therefore by lemma~\ref{minusone} $c\sim c'$ for a simple correspondence $c'$ of level $n-1$. The lemma follows by induction.
\end{proof}
This allows us to state the main result of the section:

\begin{lem}\label{Moving}
Let $Y$ be an open subset of an affine space $\A^m$. Then any correspondence $c\in Fr_n(k,Y)$ is equivalent to a difference of simple correspondences in $Fr_1(k,Y)$ in  $H_0(\ZZ F(\Delta^{\bullet}_k,Y)).$
\end{lem}
\begin{proof}
Consider a framed correspondence $c$. In $\ZZ F_n(\Spec k,Y)$ it is equivalent to a sum of correspondences of the form $c'=(U,\phi,f)$ such that $\supp(c')$ is a single point.
By lemma~\ref{simple_reduction} every $c'$ is equivalent to a difference $c''-c'''$ of simple correspondences. By lemma~\ref{one} both $c''$ and $c'''$ are equivalent to some simple correspondences in $Fr_1(k,Y)$.
\end{proof}

\section{Level one correspondences}\label{LevelOneSection}
\begin{dfn}
We will call a correspondence $c$ in $\ZZ F(\Spec k,Y)$ standard if it is given by the sum of correspondences of the form
$(\A^1_k,\lambda x,\mu\circ pr_k)$ for some $\lambda\in k^{\times}$ and $\mu\in Y(k).$
\end{dfn}

\begin{lem}\label{red1}
Let $c=(U,\phi,f)\in Fr_1(k,Y)$ be a simple level one correspondence with $Z=\supp(c)$ a single point and a projection $U\to Z$. Then there are $\mu\in k(Z)^{\times}$ and $\lambda\in k(Z)$ such that $k(Z)=k(\lambda)$ and
$c$ is equivalent in $H_0(\ZZ F(\Delta^{\bullet}_k,Y))$ to the correspondence represented by the data $(\A^1_{k(Z)},\mu(t-\lambda),f(Z))$ where $t$ denotes the coordinate on $\A^1_{k(Z)}$ and $f(Z)$ is the constant function on $\A^1_{k(Z)}.$
\end{lem}
\begin{proof}
Take $\lambda=\alpha(Z)\in k(Z).$ Since $c$ is simple, we have that $k[U]/\phi=k(Z)$ and the composition $Z\to U\to\A^1$ is a closed embedding, so $k[t]\stackrel{\alpha^*}\to k[U]/\phi=k(Z)$ is a surjection. Hence $k(Z)=k(\lambda).$ Since the residue $\overline{\phi}$ generates $\Ms_Z/\Ms_Z^2$ over $k(Z),$ we have that $\overline{\phi}=\mu\overline{\alpha-\lambda}$ for some $\mu\in k(Z)^{\times}.$ Therefore $\phi=(1+\delta)(\mu(\alpha-\lambda))$ for some $\delta\in\Ms_Z.$ Therefore, by lemma~\ref{multiplication}
we have that $c$ is equivalent to $c'=(\alpha,U,Z,\mu(\alpha-\lambda),f).$ Consider the map $\beta\colon U\to\A^1_{k(Z)}$ that consists of $\alpha$ and the projection $U\to Z.$ Note that $\beta$ is \'etale, $\pi\circ\beta=\alpha$ and $\mu(t-\lambda)\circ\beta=\mu(\alpha-\lambda).$ This
with Lemma~\ref{anyf} gives an equivalence between $c'$ and $(\A^1_k(Z),\mu(t-\lambda),f(Z)).$
\end{proof}

\begin{lem}\label{red2}
Consider a correspondence $c=(\A^1_{k(Z)},\mu(x-\lambda),pr_k)\in Fr_1(k,k)$ with $\lambda\in k(Z)$ such that $k(Z)=k(\lambda)$ and $\mu\in k(Z)^{\times}.$ Let $p\in k[x]$ denote the minimal polynomial of $\lambda.$ Then there is a polynomial $q\in k[x]$ with $\deg(q) <\deg(p)$, and an open neighborhood $U=\A^1_k\setminus Z(q)$ in $\A^1$ such that $c$ is equivalent to the correspondence represented by $(U,q(x)p(x),pr_k)\in Fr_1(k,k)$.
\end{lem}
\begin{proof}
Note that $p$ is separable and its residue generates $\Ms_Z/\Ms_Z^2.$ Then $\mu(x-\lambda)=\frac{\mu}{p'(\lambda)}p(t)(1+m)$ for some $m\in\Ms_Z.$ Then by lemma~\ref{multiplication} we can replace $c$ with $c'=(U',\frac{\mu}{p'(\lambda)}p(t),pr_k)$ for some open neighborhood $U'$ of $Z$ in $\A^1_{k(Z)}.$ Further, $k(\lambda)=k[t]/p(t),$ therefore $\frac{\mu}{p'(\lambda)}=q(\lambda)$ for some polynomial $q\in k[t]$ with $\deg(q)<\deg(p).$ Denote by $\delta=q(\lambda)-q(t)\in\Ms_Z.$ Then we have
\[\frac{\mu}{p'(\lambda)}p(t)=(1+\frac{\delta}{q(t)})q(t)p(t).\]
Thus we have that $c''$ is equivalent to the correspondence $(U,q(x)p(x),pr_k)\in Fr_1(k,k)$ where $U$ is the image of $U'$ under the \'etale map $\A^1_{k(Z)}\to\A^1_k.$
\end{proof}

\begin{lem}\label{kk}
Any correspondence in $Fr(\Spec k,\Spec k)$ is equivalent to a standard correspondence.
\end{lem}
\begin{proof}
By Lemma~\ref{Moving},~\ref{red1} and~\ref{red2} any correspondence in $Fr(\Spec k,\Spec k)$ is equivalent to a sum of correspondences of the form $c=(\A^1\setminus Z(q),qp(x),pr_k)$ for some $q,p\in k[x]$ with $\deg(q)<\deg(p)$. We proceed by induction on $\deg(c).$ Take $c'=(\A^1\setminus Z(p),qp(x),pr_k)$.
Note that $c+c'=(\A^1,qp(x),pr_k).$ Since $\deg(c')=\deg(q)<\deg(p)$ by induction hypothesis it is sufficient to prove the statement for the correspondence $(\A^1,qp(x),pr_k).$ There is a polynomial $p_1(x)$ with the same leading term as $qp(x)$ and all roots of $p_1(x)$ are rational and separable. The lemma follows then by~\ref{ratroots}.
\end{proof}

\begin{lem}\label{ratsupp}
Suppose $c=(U,\phi,f)$ represents a framed correspondence in $Fr_n(k,Y)$ and $\supp(c)$ consists of rational points in $\A^n_k$. Then $c$ is equivalent to a standard correspondence.
\end{lem}
\begin{proof}
We may assume that $Z=\supp(c)$ is a single point. Then $f(Z)\in Y(k).$ Then by Lemma~\ref{anyf} $c$ is equivalent to the composition $c=(U,\phi,pr_k)\circ f(Z)$ where $f(Z)\colon\Spec k\to Y$ is the rational point inclusion and $(U,\phi,f)\in Fr_n(k,k)$. The lemma follows then from~\ref{kk}.
\end{proof}

\section{Map $\Phi\colon H_0(\ZZ F(\Delta^{\bullet}_k,\Gm^{\wedge*}))\to K_*^{MW}(k)$}\label{cortok}
In this section we construct a homomorphism 
$\Phi\colon H_0(\ZZ F(\Delta^{\bullet}_k,\Gm^{\wedge*}))\to K_*^{MW}(k).$
(see subsection~\ref{phiconst}). To do that we need certain preliminaries.
Suppose the data $c=(U,\phi,f)$ represents a correspondence in $Fr_n(k,\Gm^{\times m})$. Let $Z=\supp(c).$
Then the sequence $\phi$ defines a Koszul complex $K(\phi)$ which can be considered as an element of $W^n_Z(U)$

We will use the theory of Chow-Witt groups introduced by Barge-Morel and developed by Fasel~(\cite{Fasel},\cite{FaselThesis}).
Recall(~\cite{Fasel}) that for a smooth scheme $X\in\Sm_k$, integer $n$ and a line bundle $L$ over $X$ there is a complex $C(X,G^n,L)$
defined by
\[
C^m(X,G^n,L)=\coprod_{x\in X^{(m)}}K_{n-m}^M(k(x))\times_{\frac{I^{n-m}}{I^{n-m+1}}}I^{n-m}_{fl}(\Os_{X,x},L)
\]
For a closed subset $Z$ and its complement $U=X\setminus Z$ define a subcomplex $C(X,G^n,L)_Z$ as the kernel of the map $C(X,G^n,L)\to C(U,G^n,L),$ so there is a short exact sequence
\[0\to C(X,G^n,L)_Z\to C(X,G^n,L)\to C(U,G^n,L)\to 0\]

\subsection{Two commutative squares}\label{comsq}
Now let $d=(U,\phi,f)\in Fr_n(\A^1,\Gm^m)$ be a framed correspondences. Let $\alpha$ be the \'etale map $\alpha\colon U\to\A^n_{\A^1}$ and $j\colon \A^n_{\A^1}\to\PP^n_{\A^1}$ be an open embedding. The composition $j\circ\alpha\colon U\to\A^n_{\A^1}\to\PP^n_{\A^1}$ induces a morphism of complexes (\cite[\S 3]{Fasel},\cite[Corollary 10.4.2]{FaselThesis})
\[j^*\circ\alpha^*\colon C(\PP^n_{\A^1},G^{n+m},\Os(-n-1))\to C(U,G^{n+m}).\]
Denote $Z=\supp(d)$. For any point $z\in Z$ the map $\alpha$ induces an isomorphism $W_{fl}(\Os_{\PP^n_{\A^1},z})\to W_{fl}(\Os_{U,z})$. Then we get an isomorphism of complexes
\[j^*\circ\alpha^*\colon C(\PP^n_{\A^1},G^{n+m},\Os(-n-1))_{Z}\stackrel{\cong}\to C(U,G^{n+m})_{Z}.\]
Let $i_0\colon U_0\to U$ denote the inclusion of zero fiber of the composition map $U\to\A^n_{\A^1}\to\A^1$ and $I_0\colon \{0\}\to\A^1$ be the standard inclusion.
Note that the oriented Gysin map defined in~\cite[Definition 5.5]{Fasel} gives rise to the Gysin map with support
$i_0^{!}\colon H^n(C(U_t,G^{n+m})_{Z_t})\to H^n(C(U_0,G^{n+m})_{Z_0})$ where $Z_0$ is the support of the zero fiber correspondence $(U_0,\phi|_{U_0},f|_{U_0})$
Note the following
\begin{rem}\label{diag1}
The following diagram commutes
\[
\xymatrix{
H^n(C(U_t,G^{n+m})_{Z})\ar[r]^{i_0^{!}} & H^n(C(U_0,G^{n+m})_{Z_0})\\
H^n(C(\PP^n_{\A^1},G^{n+m},\Os(-n-1))_{Z})\ar[u]^{j^*\circ\alpha^*}\ar[r]^{I_0^{!}} & H^n(C(\PP^n,G^{n+m},\Os(-n-1))_{Z_0})\ar[u]^{j_0^*\circ\alpha_0^*}
}
\]
\end{rem}
\begin{proof}
The remark follows from the fact that the oriented Gysin map commutes with flat pullback by~\cite[Lemma 5.7]{Fasel}.
\end{proof}
Recall that for a proper morphism $f$ there is a notion of pushforward map $f_*$ (see~\cite[Remark 3.36]{Fasel},\cite[Chapitre 8]{FaselThesis}).

\begin{lem}\label{diag2}
The following diagram commutes:
\[
\xymatrix{
H^n(C(\PP^n_{\A^1},G^{n+m},\Os(-n-1))_{Z_t})\ar[d]^{p_*}\ar[r]^{I_0^{!}} & H^n(C(\PP^n,G^{n+m},\Os(-n-1))_{Z_0})\ar[d]^{p_*}\\
H^0(C(\A^1,G^{m}))\ar[r] & H^0(C(\Spec k,G^{m},))
}
\]
Here $p$ denotes the projection and the bottom arrow is the Gysin map for the embedding $\{0\}\to \A^1.$
\end{lem}
\begin{proof}
Denote $\Os(-n-1)$ by $L$ for brevity.
Recall that the Gysin morphism $I_0^{!}$ is defined as the composition
\[
H^n(C(\PP^n_{\A^1},G^{n+m},L)_{Z_t})\to H^n(C(\PP^n_{\A^1\setminus 0},G^{n+m},L)_{Z_t\setminus Z_0})
\stackrel{\cdot\{t\}}\to H^n(C(\PP^n_{\A^1\setminus 0},G^{n+m+1},L)_{Z_t\setminus Z_0})\stackrel{\partial}\to
\]
\[\to H^{n+1}(C(\PP^n_{\A^1},G^{n+m+1},L)_{Z_0})\stackrel{\cong}\to H^n(C(\PP^n,G^{n+m},L)_{Z_0}).
\]
We have to show that $p_*$ commutes with each arrow in this decomposition.

First, $p_*$ commutes with flat pullbacks by~\cite[Corollaire 12.3.7.]{FaselThesis}

Second, $p_*$ commutes with multiplication by $\{t\}$ by the projection formula which can be derived from the definition of $p_*$ and projection formula for the
transfer map for finite field extensions: Here we fix once a trivialisation of the canonical bundle on $\A^1.$
Following~\cite[\S 8]{FaselThesis} there map $p_*\colon C^i(\PP^n_{\A^1\setminus 0},W,L)\to C^{i-n}(\A^1\setminus 0,W)$ is defined as follows:
for every $x\in{\PP^n_{\A^1}}^{(i)}$ take $y=f(x)$ and define $\theta_x^y\colon W(k(x),L_x)\to W(k(y))$ using the Scharlau transfer map for finite extension when $k(x)$ is finite over $k(y)$ and set $\theta_x^y=0$ when $k(x)$ is an infinite extension of $k(y).$ This transfer map satisfies the projection formula property, as well as transfer for Milnor $K$-theory. Therefore induced map on cohomology satisfies the projection formula property.

Third, by~\cite[Corollaire 10.4.5.]{FaselThesis} the projection map $\PP^n_{\A^1}\to\A^1$ induces morphism of complexes $C^*(\PP^n_{\A^1},G^{n+m},L)\to C^{*-n}(\A^1,G^{m})$. Then it induces a morphism of short exact sequences of complexes:
\[
\xymatrix{
C^*(\PP^n_{\A^1},G^{n+m+1},L)_{Z_0}\ar@{^{(}->}[r]\ar[d]^{p_*} & C^*(\PP^n_{\A^1},G^{n+m+1},L)_{Z_t}\ar@{->>}[r]\ar[d]^{p_*} & C^*(\PP^n_{\A^1\setminus 0},G^{n+m+1},L)_{Z_t\setminus Z_0}\ar[d]^{p_*}\\
C^{*-n}(\A^1,G^{m+1})_{\{0\}}\ar@{^{(}->}[r] & C^{*-n}(\A^1,G^{m+1})\ar@{->>}[r] & C^{*-n}(\A^1\setminus 0),G^{m+1})
}
\]
Therefore $p_*$ commutes with the connecting homomorphism $\partial$ in the corresponding long exact sequence of cohomologies:
\end{proof}

\subsection{Map construction}\label{phiconst}
Using the results of~\ref{comsq} we will define a map
\[\Phi\colon H_0(\ZZ F(\Delta^{\bullet}_k,\Gm^{m}))\to K_m^{MW}(k).\]
Let $c=(U,\phi,f)\in Fr_n(X,\Gm^{m})$ represent a framed correspondence of level~$n$. Suppose that $X$ is affine. 
We will be interested in the cases $X=\A^1$ and $X=\Spec k$ so assume that there is a fixed trivialization of a canonical bundle on $X$.
Further we may assume that $U$ is affine.
Let $Z=\supp(c)$. Note that $Z=\cup \overline{\{z_i\}}$ for some points codimension $n$ points $z_1,\ldots,z_d\in U^{(n)}$.

The framing $\phi$ defines its Koszul complex $K(\phi):$
\[K_p(\phi)=\wedge^{p}(\oplus_{i=1}^nk[U]e_i), d_p\colon K_p(\phi)\to K_{p-1}(\phi)\]
\[d_p(e_{i_1}\wedge\ldots\wedge e_{i_p})=\sum_{j=1}^p (-1)^jf_{i_j}e_{i_1}\wedge\ldots\wedge e_{i_{j-1}}\wedge e_{i_{j+1}}\wedge\ldots\wedge e_{i_p}.\]
The complex $K(\phi)$ is a free resolution of $k[U]/(\phi_1,\ldots,\phi_n).$ We consider $K(\phi)$ as an element of the bounded derived category with support $D^b_Z(U).$
Endow $K(\phi)$ with the structure of the quadratic space:
\[\theta\colon K_p(\phi)\to Hom_{k[U]}(K_{n-p}(\phi),k[U]), \theta(x)\colon y\mapsto x\wedge y.\]
We will use the notation of~\cite[Definition 3.1]{Fasel}
Thus we  consider $K(\phi)$ as an element of $D^b_Z(U)\subseteq D^b(U)^{(n)}$ which defines a quadratic space $(K(\phi),\theta)\in W^n(D^b(U)^{(n)})$.
Let $K_c$ be the image of $(K(\phi),\theta)$ under the composition where the second map is the devissage isomorphism (see~\cite[Proposition 3.3]{Fasel})
\[W^n(D^b(U)^{(n)})\to W^n(D^b_n(U))\cong \oplus_{x\in U^{(n)}}W_{fl}(\Os_{U,x}).\]
Note that $K_c$ lies in the sum $\oplus_{j=1}^dW_{fl}(\Os_{U,z_j})$. The latter is canonically isomorphic to $\oplus_{j=1}W(k(z_j),\omega_{z_j/U})$
where $\omega_{z_j/U}\cong\bigwedge^n(\Ms_{U,z_j}/\Ms_{U,z_j}^2)$ (see~\cite[(6)]{BW}). The \'etale map $\alpha$ and trivialization of the canonical bundle on $\A^n_X$ gives an isomorphism $\bigwedge^n(\Ms_{U,z_j}/\Ms_{U,z_j}^2)\stackrel{\alpha^*}\cong (\Ms_{\A^n_X,z_j}/\Ms_{\A^n_X,z_j}^2)\cong k(z_j).$ Then for any $j=1..d$ we get the pair $y_j=(l_{\Os_{U,z_j}}(k[U]/(\phi)),K_c)\in \mathbb{Z}\times_{\mathbb{Z}/2}W(k(z_j))=GW(k(z_j)).$ The elements $f(z_j)=(a_{1,j},\ldots a_{m,j})\in\Gm^m(k(z_j))$
this gives element \[x_j=y_j\cdot[a_{1,j}]\ldots[a_{m,j}]\in K_m^{MW}(k(z_j)).\]
Note that since $\phi$ and $f$ are defined on $U$, this element has a zero residue for every point $x\in\overline{z_j}^{(1)}$, thus $x_c=\sum_{j=1}^d x_j$ defines an element in $H^n(C(U,G^{n+m})_Z)$.

Then we set $\Phi(c)$ to be the image of $y$ under the composition
\[H^n(C(U,G^{n+m})_Z)\stackrel{(j^*\circ\alpha^*)^{-1}}\to H^n(C(\PP^n_X,G^{n+m},L)_Z)\stackrel{p_*}\to H^0(C(X,G^m)).\]

Note that if we take an equivalent presentation of $c$ as $c'=(U',\phi',f')$ then the image of $y_{c'}\in H^n(C(U,G^{n+m})_Z)$ in $H^n(C(\A^n,G^{n+m})_Z)$ will coincide
with the image of $x_c$, sot his definition does not depend on the choice of presentation of the correspondence $c$.

Let us check that this definition is compatible with the stabilization map. Take a correspondence given by the data $c'=(U\times\A^1,(\phi,x),f\circ pr_U)$ in $Fr_{n+1}(k,\Gm^{\times m})$. Note that the Koszul complex $K((\phi,x))_i$ equals to $K(\phi)_i$ in $W(L_i)$ and for the corresponding element $y'\in H^{n+1}(C(U\times\A^1,G^{n+m+1})_Z)$ we have the diagram and the image of $x_{c'}$ under the composition $H^{n+1}(C(\A^{n+1}_X,G^{n+m+1},L)_Z)\stackrel{\cong}\to H^{n+1}(C(\PP^1\times\A^n_X,G^{n+m+1},L)_Z)\stackrel{p_*}\to H^{n}(C(\A^n_X,G^{n+m},L)_Z)$ equals to the image of $x_c$ under the isomorphism $H^{n}(C(U,G^{n+m})_Z)\to H^{n}(C(\A^{n},G^{n+m})_Z)$

The results of~\ref{comsq} allows us to check that this definition is compatible with the homotopy:
When $d=(U,\phi,f)\in Fr_n(\A^1,\Gm^{\times m})$, we get an element $x_d\in H^n(C(U,G^{n+m})_{Z})$. The element $x_d$ specialized at $0$ and $1$ gives the elements $x_{d_0}$ and $x_{d_1}$ where $d_0$ and $d_1$ in $Fr_n(k,\Gm^{\times m})$ are the fibers of $d$ over $0$ and $1$ respectively.
Then by remark~\ref{diag1} and~\ref{diag2} and homotopy invariance of Chow-Witt groups, we have that images of $x_{d_0}$ and $x_{d_1}$ are equal in $H^0(C(\Spec k,G^m))=K_m^{MW}(k).$

Therefore $\Phi$ descends to a map
\[\Phi_m\colon H_0(\ZZ F(\Delta^{\bullet}_k,\Gm^{m}))\to K_m^{MW}(k).\]

\begin{lem}\label{Phiongen}
For $a_1,\ldots, a_m\in k^{\times}$
$\Phi_m([x-a_1]\cdot\ldots\cdot[x-a_m])=[a_1]\cdot\ldots\cdot[a_m]$
\end{lem}
\begin{proof}
The correspondence $[x-a_1]\cdot\ldots\cdot[x-a_m]$ is given by the correspondence $c=(\Gm^m,(x_1-a_1,\ldots,x_m-a_m),id)\in Fr_m(k,\Gm^{m}).$ Then its Koszul complex gives trivial element in $GW(k)$. Then $y_c=([a_1]\cdot\ldots\cdot[a_m])$, so $\Phi_m(c)=[a_1]\cdot\ldots\cdot[a_m].$
\end{proof}

\section{The case $m=0$}\label{m0Section}
In this section we prove~(\ref{m0}) that $\Phi_0$ induces an isomorphism between $H_0(\ZZ F(\Delta^{\bullet}_k,\Spec k))$ and $K_0^{MW}(k)=GW(k).$

For $p(x)\in k[x]$ let us denote the class of correspondence $(\A^1_k,p(x),pr_k)$ in $H_0(\ZZ F(\Delta^{\bullet}_k,\Spec k))$ by $\langle p(x)\rangle.$

\begin{lem}\label{deform0}
Suppose $p(x),q(x)\in k[x]$ be the polynomials with the same leading term. Then $\la p(x)\ra=\la q(x)\ra$
\end{lem}
\begin{proof}
We have $q(x)=p(x)+\delta(x)$ with $\deg(\delta)<\deg(p).$ Then the data $(\A^2,p(x)+\delta(x)t,pr_k)$ defines a correspondence in $Fr_1(\Spec k[t],\Spec k)$ which gives a deformation between $\la p\ra$ and $\la q\ra.$
\end{proof}

\begin{rem}
By Lemma~\ref{multiplication} for a polynomial $p$ with rational separable roots $x_i, \ i=1,\ldots,n$ we have $\langle p(x)\rangle=\langle p'(x_1)x\rangle+\ldots \langle p'(x_n)x\rangle$ where $p'$ is the derivative of $p.$
\end{rem}

\begin{lem}\label{squar}
For any $\lambda,c\in k^{\times}$ the following holds: $\langle cx\rangle=\langle c\lambda^2x\rangle$ and $\langle cx\rangle+\langle -cx\rangle=\langle x\rangle+\langle -x\rangle.$
\end{lem}
\begin{proof}
By~\ref{ratroots} for any $\lambda,a\in k^{\times}$ we have $\langle ax^2\rangle=\langle ax^2-a\lambda^2\rangle=\langle 2a\lambda x\rangle+\langle -2a\lambda x\rangle.$ Thus $\langle x^2\rangle=\langle 2x\rangle+\langle -2x\rangle=\langle ax^2\rangle$. Lemma~\ref{ratroots} implies $\langle cx^3\rangle=\langle cx^3+c\lambda x^2\rangle=\langle cx^2(x+\lambda)\rangle=\langle c\lambda^2 x\rangle+\langle c\lambda x^2\rangle.$
Then
\[\langle c\lambda^2x\rangle=\langle cx^3\rangle-\langle c\lambda x^2\rangle=\langle cx^3\rangle-\langle cx^2\rangle=\langle cx\rangle.\]

\end{proof}

\begin{prop}\label{sum}
Suppose $a_1,a_2,b_1,b_2\in k^{\times}$ such that
\begin{itemize}
\item $a_1+a_2=b_1+b_2$,
\item $a_1a_2=\alpha^2 b_1b_2$ for some $\alpha\in k^{\times}$.
\end{itemize}
Then $\langle a_1x\rangle+\langle a_2x\rangle=\langle b_1x\rangle+\langle b_2x\rangle$.
\end{prop}
\begin{proof}
Consider a polynomial $p(x)$ such that $p(0)=p(1)=0,$ $p'(0)=a_1$, $p'(1)=a_2$.
We can take $p=x(x-1)((a_1+a_2)x-a_1)=x(x-1)((b_1+b_2)x-a_1)$. Denote by $d$ the element $d=a_1+a_2=b_1+b_2$.

By~\ref{ratroots} $\langle x(x-1)(dx-a_1)\rangle= \langle x(x-1)(dx-b_1)\rangle.$
On the other hand
$\langle x(x-1)(dx-a_1)\rangle=\langle a_1x\rangle+\langle a_2x\rangle+\langle \frac{-a_1a_2}{d}x\rangle$ and
$\langle x(x-1)(dx-b_1)\rangle=\langle b_1x\rangle+\langle b_2x\rangle+\langle \frac{-b_1b_2}{d}x\rangle.$

Since $a_1a_2=\alpha^2b_1b_2$, Lemma~\ref{squar} implies that $\langle \frac{-a_1a_2}{d}x\rangle=\langle \frac{-b_1b_2}{d}x\rangle$. Then $\langle a_1x]+\langle a_2x\rangle=\langle b_1x\rangle+\langle b_2x\rangle$
\end{proof}

\begin{lem}\label{inverse}
There is a surjective homomorphism $\Psi_0\colon K_0^{MW}(k)=GW(k)\to H_0(\ZZ F(\Delta^{\bullet}_k,\Spec k))$ such that $\Psi_0\colon\langle a\rangle\mapsto\langle ax\rangle$ for any $a\in k^{\times}$
\end{lem}
\begin{proof}
According to the lemma~\cite[2.9]{Morel} the Grothendieck-Witt group is generated by the 1-forms $\langle a \rangle$ modulo the relations:
\begin{itemize}
\item $\langle ab^2\rangle=\langle a\rangle$,
\item $\langle a \rangle + \langle -a \rangle = \langle 1 \rangle + \langle -1 \rangle$,
\item $\langle a\rangle + \langle b\rangle = \langle a+b \rangle + \langle (a+b)ab \rangle$ if $a+b\neq 0.$
\end{itemize}
By lemma~\ref{squar} in $H_0(\ZZ F(\Delta^{\bullet}_k,\Spec k))$ we have $\langle ab^2x\rangle=\langle ax\rangle$ and $\langle cx\rangle+\langle -cx\rangle=\langle x\rangle+\langle -x\rangle.$ Lemma~\ref{sum} implies the third relation $\langle ax\rangle+\langle bx\rangle=\langle a(1+b)^2x\rangle+\langle b(1+a)^2x\rangle=\langle (a+b)x\rangle+\langle (a+b)abx\rangle.$
Thus the assignment $\langle a \rangle\mapsto \langle ax\rangle$ gives a well-defined homomorphism $F.$ It is surjective by~\ref{kk}.
\end{proof}

\begin{prop}\label{m0}
The maps $\Psi_0$ and $\Phi_0$ are mutually inverse ring isomorphisms between $H_0(\ZZ F(\Delta^{\bullet}_k,\Spec k))$ and $K_0^{MW}(k).$
\end{prop}
\begin{proof}
Check that $\Phi_0(\langle ax\rangle)=\langle a\rangle$. The Koszul complex $K(ax)=(0\to k[x]\stackrel{\cdot ax}\to k[x]\to 0)$. It defines an element in $W^1_{\{0\}}(\A^1)\cong W(k).$
According to the results of Gille~\cite[Definition 8.2]{Gille} the devissage isomorphism $\nu\colon W(k)\to W^1_{\{0\}}(\A^1_k)$ is given by $y\mapsto s^*(y)\star K(x)$
where $K(x)$ is the Koszul complex of framing $x$.
Then $\nu(\langle a\rangle)$ is given by 
\[
\xymatrix{
k[x] e\ar[d]^{\cdot x} \ar[r]^{\cdot a} & k[x]\ar[d]^{\cdot(-x)} &=Hom_{k[x]}(k[x],k[x])\\
k[x] \ar[r]^{\cdot(-a)} & k[x]&=Hom_{k[x]}(k[x] e,k[x])
}
\]
is isomorphic to the quadratic space $K(ax)$ by means of isomorphism
\[
\xymatrix{
k[x]\ar[d]^{\cdot x} \ar[r]^{\cdot 1} & k[x]\ar[d]^{\cdot a x}\\
k[x] \ar[r]^{\cdot a} & k[x]
}
\]
Thus $K(ax)$ is mapped to $\langle a\rangle\in W(k)\subset H^1(C(\A^1,G^1))$, therefore $\Phi_0(\langle ax\rangle)=\langle a\rangle.$ Thus $\Phi_0\circ\Psi_0=id_{GW(k)}$ and $\Psi_0$ is surjective by~\ref{inverse}. Then $\Phi_0$ and $\Psi_0$ are mutually inverse isomorphisms. It remains to show that $\Psi_0$ is a ring isomorphisms. Indeed, $\Psi_0(\la ab\ra)=\la abx\ra=(\A^1,abx,pr_k)\sim(\A^2,(abx,x_2),pr_k)\stackrel{\ref{SL}}\sim(\A^2,(ax,bx_2),pr_k)=\la ax\ra\cdot\la bx\ra=\Psi_0(\la a\ra)\cdot \Psi_0(\la b\ra).$
\end{proof}

\begin{dfn} Following~\cite[Lemma 2.14]{Morel} we denote $n_{\epsilon}=\sum_{i=1}^n\la(-1)^{i-1}\ra$ for $n\geqslant 0$ and $n_{\epsilon}=-\la-1\ra(-n)_{\epsilon}$ for $n<0.$
\end{dfn}

\begin{rem}
$\Phi_0(\la x^n\ra)=n_{\epsilon}$.
\end{rem}
\begin{proof}
By~\ref{m0} $\Phi_0(\la x\ra)=\la 1\ra$.
By induction $\Phi_0(\la x^n\ra)\stackrel{\ref{deform0}}=\Phi_0(\la x^{n-1}(x+1)\ra)=(n-1)_{\epsilon}+\la(-1)^{n-1}\ra.$
\end{proof}

\begin{lem}\label{trcomp}
Suppose $p(x)\in k[x]$ and $U$ be an open subset of $\A^1_k$ and $p(x)=q(x)(x-\lambda_1)^{r_1}\cdot\ldots\cdot(x-\lambda_k)^{r_k}$ where $q(x)$ is invertible over $U$ and $\lambda_i\in k.$ Let $c$ be a class of correspondence in $Fr_1(k,k)$ given by data $(U,p(x),pr_k)$. Then we have in $K_0^{MW}(k)$
\[
\Phi_0(c)=\sum_{i=1}^k(r_i)_{\epsilon}\la q(\lambda_i)\prod_{j\neq i}(\lambda_i-\lambda_j)^{r_j}\ra=-\partial^{-1/x}_{\infty}(\sum_{i=1}^k[(x-\lambda_i)^{r_i}\cdot q(\lambda_i)\prod_{j\neq i}(\lambda_i-\lambda_j)^{r_j}]).\]
\end{lem}
\begin{proof}
By~\ref{multiplication} we have $(U,p(x),pr_k)=\sum_{i=1}^k(U_i,p(x),pr_k)\sim\sum_{i=1}^k\la q(\lambda_i)(x-\lambda_i)^{r_i}\prod_{j\neq i}(\lambda_i-\lambda_j)^{r_j}\ra$ where $U_i$ is the neighborhood of $\lambda_i$ that does not contain any $\lambda_j$ for $j\neq i$. Since $\Phi_0(\la x^n\ra)=n_{\epsilon}$ in $K_0^{MW}(k)$ we get the first equality. The second follows from the fact that  for $u\in k^{\times}$ we have $-\partial^{-1/x}_{\infty}([(x-\lambda_i)^{r_i}u])=-((-r_i)_{\epsilon}\la(-1)^{r_i}u\ra)=\la(-1)^{r_i+1}\ra(r_i)_{\epsilon}\la u\ra=(r_i)_{\epsilon}\la u\ra.$
\end{proof}

\begin{lem}\label{com0} Suppose $[L:k]=l$ is a prime number and $k$ has no prime-to-$l$ extensions.
Then the transfer diagram is commutative
\[
\xymatrix{
K_0^{MW}(L)\ar[rr]^{\Psi_0}\ar[d]^{Tr^L_k} && H_0(\ZZ F(\Delta^{\bullet}_L,\Spec L))\ar[d]^{tr_{L/k}}\\
K_0^{MW}(k)\ar[rr]^{\Psi_0} && H_0(\ZZ F(\Delta^{\bullet}_k,\Spec k))
}
\]
\end{lem}
\begin{proof}
Take $\alpha\in L$. Since $l$ is prime, we may assume that $L=k(\alpha).$ Then $tr_{L/k}(\Psi_0(\la\alpha\ra))=tr_{L/k}(\la\alpha x\ra)\sim(\A^1_L,\alpha(x-\alpha),pr_k).$ Let $p(x)\in k[x]$ be the minimal monic polynomial for $\alpha.$ Then $(\A^1_L,\alpha(x-\alpha),pr_k)\sim(U',xp'(x)p(x),pr_k)$ where $U'$ is an open subset of $\A^1_L$ that does not contain any root of $xp'(x)p(x)$ except $\alpha.$ Take $U\subseteq\A^1_k$ to be the image of $U'.$ Then $tr_{L/k}(\Psi_0(\la\alpha\ra))=(U',xp'(x)p(x),pr_k)=(U,xp'(x)p(x),pr_k)=\la xp'(x)p(x)\ra - (\A^1_k\setminus Z(p), xp'(x)p(x),pr_k)\stackrel{\ref{deform0}}=\la x^{2l}\ra- (\A^1_k\setminus Z(p), xp'(x)p(x),pr_k).$ Since $k$ has no prime-to-$l$ extensions, all roots of $p'(x)$ are rational. So write $xp'(x)=l(x-\lambda_1)^r_1\ldots(x-\lambda_k)^{r_k}$. Then by~\ref{trcomp}
\[
\Phi_0(tr_{L/k}\Psi_0(\la\alpha\ra))=\la l\ra(2l)_{\epsilon}-\left(\sum_{i=1}^k(r_i)_{\epsilon}\la lp(\lambda_i)\lambda_i\prod_{j\neq i}(\lambda_i-\lambda_j)^{r_j}\ra\right).
\]
From the other hand, $Tr^L_k(\la\alpha\ra)=\tau^L_k(\alpha)(\la p'(\alpha)\alpha\ra)=-\partial_{\infty}^{-1/x}([xp'(x)p(x)]-\sum_{i=1}^k[(x-\lambda_i)^{r_i}lp(\lambda_i)\lambda_i\prod_{j\neq i}(\lambda_i-\lambda_j)^{r_j}])=\la l\ra(2l)_{\epsilon}-\left(\sum_{i=1}^k(r_i)_{\epsilon}\la lp(\lambda_i)\lambda_i\prod_{j\neq i}(\lambda_i-\lambda_j)^{r_j}\ra\right)$ by~\ref{trcomp}. Thus $\Phi_0(tr_{L/k}\Psi_0(\la\alpha\ra))=Tr^L_k(\alpha),$ so the diagram commutes.
\end{proof}

Further we will identify $GW(k)$ with $H_0(\ZZ F(\Delta^{\bullet}_k,\Spec k)).$ Then the multiplication structure of~\ref{MultStruct} endows $H_0(\ZZ F(\Delta^{\bullet}_k,\Gm^{\wedge *}))$ with the structure of $GW(k)$-module.

\begin{lem}\label{transfer1}
Suppose that $L$ is a degree $n$ extension of $k$. Then $\Phi_0(tr_{L/k}(\la x\ra))=n_{\epsilon}\in K_0^{MW}(k).$
\end{lem}
\begin{proof}
Take a generator $\alpha$ with minimal monic polynomial $p(x)\in k[x].$ Consider the Gaussian algorithm $p=r_0, p'=r_1, r_i=r_{i-2}\mod r_{i-1}.$ Let $d_i=\deg(r_i).$ Since $p$ and $p'$ are coprime, $d_N=0$ for some $N$. Then $\Phi_0(tr_{L/k}(\la x \ra))=\Phi_0((\A^1_k\setminus Z(p'),p'(x)p(x),pr_k))=\Phi_0(\la p'(x)p(x)\ra-(\A^1_k\setminus Z(p),p'(x)p(x),pr_k))\stackrel{\ref{multiplication},\ref{deform0}}=(d_0+d_1)_{\epsilon}-\Phi_0((\A^1_k\setminus Z(r_2),p'(x)r_2(x),pr_k))=(d_0+d_1)_{\epsilon}-((d_1+d_2)_{\epsilon}-\Phi_0((\A^1_k\setminus Z(r_3),r_2(x)r_3(x),pr_k)))=\ldots=\sum_{i=0}^N(-1)^i(d_i+d_{i+1})_{\epsilon}=(d_0)_{\epsilon}=n_{\epsilon}$.
\end{proof}

\section{Map from Milnor-Witt K-theory to Framed correspondences.}\label{ktocor}
In this section we extend $\Psi_0$ to a morphism $\Psi_*\colon K_*^{MW}(k)\to H_0(\ZZ F(\Delta^{\bullet}_k,\Gm^{\wedge *}))$ (see subsection~\ref{psiconst}).
Recall that Milnor-Witt $K$-theory Is defined using the generators $[a]$ of degree $1$ for $a\in k^{\times}$, $\eta$ of degree $-1$ and the relations of~\cite[Definition 2.1]{Morel}

\begin{ntt}
For a polynomial $p(x)\in k[x]$ we will denote by $[p(x)]$ the class of the correspondence $(\Gm,p(x),id)\in Fr_1(\Spec k,\Gm)$ in $H_0(\ZZ F(\Delta^{\bullet}_k,\Gm^{\wedge 1}))$
\end{ntt}
\subsection{Some relations in $H_0(\ZZ F(\Delta^{\bullet}_k,\Gm^{\wedge 1}))$}

\begin{rem}\label{onesupp}
By Lemma~\ref{anyf} for any correspondence $c=(U,\phi,f)\in Fr(k,\Gm)$ such that $f(Z)=\{1\}$ its class vanishes in $\ZZ F(k,\Gm^{\wedge 1})=\coker(\ZZ F(k,k)\stackrel{i_{1*}}\to\ZZ F(k,\Gm)).$

\end{rem}

\begin{lem}\label{hout}
$[(x-a)^2]=h[x-a]$ where $h\in GW(k)$ is the hyperbolic plane
\end{lem}
\begin{proof}
$(\Gm,(x-a)^2,id)\sim (\Gm\times\A^1,((x_1-a)^2,x_2),pr_{\Gm})\sim(\Gm\times\A^1,((x_1-a)^2,x_2-(x_1-a)),pr_{\Gm})\stackrel{\ref{nilpotent}}\sim
(\Gm\times\A^1,(x_2^2,x_2-(x_1-a)),pr_{\Gm})\sim (\Gm\times\A^1,(x_2^2,-(x_1-a)),pr_{\Gm})\stackrel{\ref{SL}}\sim (\Gm\times\A^1,((x_1-a),x_2^2),pr_{\Gm})=(\Gm,x_1-a,id_{\Gm})\cdot(\A^1,x^2,pr_k).$ Since $(\A^1,x^2,pr_k)$ in $H_0(\ZZ F(\Delta^{\bullet}_k,\Spec k))$ is equivalent to the sum $\langle x\rangle+\la -x\ra$, it is mapped to $h$ in $GW(k).$
\end{proof}

\begin{lem}\label{square}
The following equality holds in $H_0(\ZZ F(\Delta^{\bullet}_k,\Gm^{\wedge 1}))$:
\[[x-a^2b]=[x-b]+h[x-a]\text{ where $h\in GW(k)$ is the hyperbolic plane}\]
\end{lem}
\begin{proof}
By Lemma~\ref{deform}  we have that $[(x-1)^2(x-a^2b)]=[(x-a)^2(x-b)]$.
By~\ref{multiplication} and Remark~\ref{onesupp} we have $[(x-1)^2(x-a^2b)]=\la (a^2b-1)^2\ra [x-a^2b]=[x-a^2b].$

For the right hand side by~\ref{hout} we have $[(x-a)^2(x-b)]=\la a-b\ra[(x-a)^2]+\la(b-a)^2\ra[x-b]=\la a-b\ra h[x-a]+[x-b]=h[x-a]+[x-b].$
\end{proof}

\begin{lem}\label{K1multiplication}
The following equality holds in $H_0(\ZZ F(\Delta^{\bullet}_k,\Gm^{\wedge 1}))$ for $a,b\in k^{\times}\setminus 1$ and $ab\neq 1$
\[\la ab-1\ra[x-ab]=\la a-1\ra[x-a]+\la ab-a\ra[x-b]\]
\end{lem}
\begin{proof} For any $a',b'\in k^{\times}$, $a'\neq b'$ consider a polynomial $p(x)=1/(a'-b')(x-1)(x-a')(x-b')$. Then $[p(x)]=\la a'-1\ra[x-a']+\la 1-b'\ra[x-b'].$ By lemma~\ref{deform} we have
\[[p(x)]=[1/(a'-b')(x-1)^2(x-a'b')]=\la 1/(a'-b')(a'b'-1)^2\ra[x-a'b']=\la a'-b'\ra[x-a'b'].\]
Thus we get the equality
\[\la a'-1\ra[x-a']+\la 1-b'\ra[x-b']=\la a'-b'\ra[x-a'b']\]
For $a'=a$ and $b'=ab$ we get $\la a-1\ra[x-a]+\la 1-ab\ra[x-ab]=\la a-ab\ra[x-a^2b]$.
By~\ref{square} this implies
$\la a-1\ra[x-a]+\la 1-ab\ra[x-ab]=\la a-ab\ra[x-b]+h[x-a]$. Multiplying by $\la -1\ra$ we get
$\la ab-a\ra[x-b]+(h-\la 1-a\ra)[x-a]=\la ab-1\ra[x-ab]$.
Since $h-\la 1-a\ra=\la a-1\ra$ we get the needed equality.
\end{proof}

\subsection{The Steinberg relation}

\begin{lem}\label{tracegen}
Let $L=k(\alpha)$ and $p(x)$ be the unital minimal polynomial of $\alpha$. Let $n=\deg p.$
Then $tr_{L/k}(\la p'(\alpha)\ra[x-\alpha])=\la (N(\alpha)-1)^{n-1}\ra[x-N(\alpha)].$
\end{lem}
\begin{proof}
We have that $tr_{L/k}(\la p'(\alpha)\ra[x-\alpha])\sim((\Gm)_L,p'(\alpha)(x-\alpha),pr_{\Gm}).$ Then by~\ref{multiplication} it is equivalent to $(\Gm,p(x),pr_{\Gm})\stackrel{\ref{deform}}\sim(\Gm,(x-1)^{n-1}(x-(-1)^np(0)))=\la ((-1)^np(0)-1)^{n-1}\ra[x-(-1)^np(0)]$ if $(-1)^np(0)\neq 1$ and $0$ otherwise. The lemma follows since $N(\alpha)=(-1)^np(0).$
\end{proof}

\begin{lem}\label{ptrace}
Let $l$ be a prime number and $a\in k^{\times}\setminus k^{\times l}$ and $L=k(\sqrt[l]{a}).$
Then $tr_{L/k}([x-(1-\sqrt[l]{a})])=\langle l\rangle[x-(1-a)]$ and $tr_{L/k}([x-\sqrt[l]{a}])=\la l\ra[x-a]$in $H_0(\ZZ F(\Delta^{\bullet}_k,\Gm^{\wedge 1})).$
\end{lem}
\begin{proof}
When $l$ is odd the lemma follows from~\ref{tracegen} applied to the polynomial $p(x)=x^l-a$. Then $p'(x)=lx^{l-1}$ and since $l-1$ is even, $\la p'(\alpha)\ra=\la l \ra$
and $\la(N(\alpha)-1)^{l-1}\ra=1.$

When $l=2$ take $f(x)=(-2)(1-x)((1-x)^2-a)$. Then $f'(1-\sqrt{a})=4a^2$. Let $U=(\Gm)_L\setminus \{1\}.$ Then $(U,f(x),pr_{\Gm})=tr_{L/k}([x-(1-\sqrt{a}])).$ Then
$[f(x)]=tr_{L/k}([x-(1-\sqrt{a})])$ in $H_0(\ZZ F(\Delta^{\bullet}_k,\Gm^{\wedge 1}))$. Since $[f(x)]=[-2(1-x)^{4}((1-x)-a)]=\langle 2\rangle[x-(1-a)]$, we get the statement of the lemma. The proof of the second statement goes analogously for $f(x)=px^{p+1}(x^p-a).$
\end{proof}

\begin{lem}\label{12}
Let $n$ be an integer and suppose that $nh[x-a][x-(1-a)]=0$ in $H_0(\ZZ F(\Delta^{\bullet}_L,\Gm^{\wedge 2}))$ for all field extensions $L/k$ and $a\in L^{\times}, a\neq 1.$
Then $h[x-a][x-(1-a)]=0$ in $H_0(\ZZ F(\Delta^{\bullet}_k,\Gm^{\wedge 2})).$
\end{lem}
\begin{proof}
Let $n=mp$ for $p$ a prime number. We will prove that $mh[x-a][x-(1-a)]=0$. Let $b=\sqrt[p]{a}$. Then
$0=mph[x-b][x-(1-b)]=mh[x-a][x-(1-b)]$ over $k(b).$ Then by~\ref{ptrace}
\[0=tr_{k(b)/k}([mh[x-a][x-(1-b)]])=mh[x-a]tr_{k(b)/k}[x-(1-b)]=mh[x-a][x-(1-a)].\]
\end{proof}

\begin{lem}\label{Steinberg}
For any $a\in k^{\times}$, $a\neq 1$ we have $[x-a]\cdot[x-(1-a)]=0$ in $H_0(\ZZ F(\Delta^{\bullet}_k,\Gm^{\wedge 2})).$
\end{lem}
\begin{proof}
The proof consists of two steps. First we prove that $h[x-a]\cdot[x-(1-a)]=0$ in $H_0(\ZZ F(\Delta^{\bullet}_L,\Gm^{\wedge 2}))$ for any field extension $L/k,$ and $a\in L^{\times}, a\neq 1$ where $h\in GW(L)$ is the hyperbolic plane.

Once this is done, the Steinberg relation follows since
$0=tr_{k(\sqrt{a})/k}(h[x-\sqrt{a}][x-(1-\sqrt{a})])\stackrel{\ref{square}}=tr_{k(\sqrt{a})/k}([x-a][x-(1-\sqrt{a})])=[x-a]\cdot tr_{k(\sqrt{a})/k}([x-(1-\sqrt{a})])\stackrel{\ref{ptrace}}=\langle 2\rangle[x-a][x-(1-a)].$

Now we check that $h[x-a]\cdot[x-(1-a)]=0.$
Here we use the same arguments as in~\cite[Proposition 5.9]{Voevodsky}.
The lemma~\ref{K1multiplication} implies that
\[h[x-ab]=h[x-a]+h[x-b]\eqno{(*)}\]
Now consider a framed correspondence in $Fr_1(\A^1,\A^1\setminus\{1,0\})$ defined as
$c=(\A^2,x^3-t(a^3+1)x^2+t(a^3+1)x-a^3,p_2)$ where $t$ is the coordinate on $A^1$ and $x$ is the coordinate on $\A^1\setminus\{1,0\}$. Composing with the embedding $\A^1\setminus\{1,0\}\to\Gm^2$, $a\mapsto (a,1-a)$ we get a correspondence $c'\in Fr(\A^1,\Gm^2)$. Let $c'_0$ and $c'_1$ be its fibers in $Fr(k,\Gm^2)$ over $0$ and $1$ respectively. Consider the field $k(\omega)$ where $\omega$ is a principal third root of unity.
Then by $(*)$ we have in $H_0(\ZZ F(\Delta^{\bullet}_{k(\omega)},\Gm^{\wedge 2}))$
\[
hc'_0=h[x-a][x-(1-a)]+h[x-\omega a][x-(1-\omega a)]+h[x-\omega^2a][x-(1-\omega^2a)]=
\]
\[=h[x-a][x-(1-a^3)]+h[x-\omega][x-(1-\omega a)(1-\omega^2a)^2]\]
and
\[
hc'_1=h[x-a^3][x-(1-a^3)]+h[x+\omega][x-(1+\omega)]+h[x+\omega^2][x-(1+\omega^2)]
\]
Since $3h[x-\omega]=h[x-1]=0$ and $3h[x-a]=h[x-a^3]$ and
$3hc'_0=3hc'_1$, we get that
\[
h[x-a^3][x-(1-a^3)]=3h[x-a^3][x-(1-a^3)]+h[x+1][x-(1+\omega)]+h[x+1][x-(1+\omega^2)].
\]
Note that $h[x+1][x-(1+\omega)]+h[x+1][x-(1+\omega^2)]=h[x+1][x-1]=0$
so $2h[x-a^3][x-(1-a^3)]=0$ in $H_0(\ZZ F(\Delta^{\bullet}_{k(\omega)},\Gm^{\wedge 2}))$.

Then by the projection formula~\ref{projection}
\[0=tr_{k(\omega)/k}(2h[x-a^3][x-(1-a^3)])=2h[x-a^3][x-(1-a^3)]tr_{k(\omega)/k}(\la 1\ra)\]
Since $tr_{k(\omega)/k}(\la 1\ra)$ is some quadratic from of degree $2$ we have $h\cdot tr_{k(\omega)/k}(\la 1\ra)=2h$
Thus $0=4h[x-a^3][x-(1-a^3)]=12h[x-a][x-(1-a^3)]$ in $H_0(\ZZ F(\Delta^{\bullet}_k,\Gm^{\wedge 2})).$

Now take $L=k(\sqrt[3]{a})$ so we have $0=12h[x-\sqrt[3]{a}][x-(1-a)]$ over $L$ and since $h tr_{L/k}([x-\sqrt[3]{a}])=h[x-a]$ by~\ref{ptrace} we get that $12h[x-a][x-(1-a)]=0$ over $k$. Then $h[x-a][x-(1-a)]=0$ by Lemma~\ref{12}.

\end{proof}

\begin{dfn}
Denote by $c_{\eta}$ the correspondence in $Fr_1({\Gm},\Spec k)$ defined by the data $(\Gm\times\A^1,xy-1,pr_k)$ where $x$ is the coordinate on ${\Gm}$ and $y$ is the coordinate on $\A^1.$
Consider the projection $p\colon\Gm\to\Spec k$ as an element of $Fr_0({\Gm},\Spec k)$. So $c_{\eta}-p$ is an element of $\ZZ F({\Gm},\Spec k).$
\end{dfn}

\begin{lem}\label{etaa}
$c_{\eta}\circ(\la\lambda\ra[x-a])=\la \lambda a\ra.$
\end{lem}
\begin{proof}
By definition, the composition $c_{\eta}\circ([x-a])$ is given by $(\Gm\times\A^1,(xy-1,\lambda(y-a)),pr_k)\sim (\Gm\times\A^1,\lambda(ax-1),y),pr_k)\sim (\A^1,\lambda ax,pr_k)=\la \lambda a\ra$.
\end{proof}

\begin{lem}\label{i1}
The composition $H_0(\ZZ F(\Delta^{\bullet}_k,\Spec k))\stackrel{i_{1*}}\to H_0(\ZZ F(\Delta^{\bullet}_k,\Gm))\stackrel{(c_{\eta}-p)_*}\to H_0(\ZZ F(\Delta^{\bullet}_k,\Spec k))$ is zero.
\end{lem}
\begin{proof}
By proposition~\ref{m0} it is sufficient to check that $(c_{\eta}-p)\circ (\la a\ra[x-1])=0$. Since $p\circ(\la a\ra[x-1])=\la a\ra,$ by~\ref{etaa} we have $(c_{\eta}-p)\circ (\la a\ra[x-1])=\la a\ra -\la a\ra=0$.
\end{proof}

\begin{lem}
The operator $((c_{\eta}-p)\times id)_*\colon H_0(\ZZ F(\Delta^{\bullet}_k,\Gm^{2}))\to H_0(\ZZ F(\Delta^{\bullet}_k,\Gm))$
descends to a map
$(c_{\eta}-p)\wedge id)_*\colon H_0(\ZZ F(\Delta^{\bullet}_k,\Gm^{\wedge 2}))\to H_0(\ZZ F(\Delta^{\bullet}_k,\Gm^{\wedge 1}))$
and for $f,g\in H_0(\ZZ F(\Delta^{\bullet}_k,\Gm^{\wedge 1}))$ we have $((c_{\eta}-p)\wedge id)_*(f\cdot g)=(c_{\eta}-p)_*(f)\cdot g.$
\end{lem}
\begin{proof}
Let $i_1\colon\Spec k\to\Gm$ be the inclusion of unity. For every $a\in H_0(\ZZ F(\Delta^{\bullet}_k,\Gm))$
we will show that $(c_{\eta}-p)\times id\circ (a\times i_1)=0$ and $(c_{\eta}-p)\times id\circ (i_1\times a)=0$ in $H_0(\ZZ F(\Delta^{\bullet}_k,\Gm^{\wedge 2}))$ Indeed, $(c_{\eta}-p)\times id\circ a\times i_1=((c_{\eta}-p)\circ a)\times i_1=0$ and $(c_{\eta}-p)\times id\circ i_1\times a=((c_{\eta}-p)\circ i_1)\times a=0$ by~\ref{i1}. The last statement follows from~\ref{prodcomp}.
\end{proof}

\begin{cor}\label{Steinbergin1}
$\la 1-a\ra[x-a]=[x-a]$ in $H_0(\ZZ F(\Delta^{\bullet}_k,\Gm^{\wedge 1})).$
\end{cor}
\begin{proof}
By~\ref{Steinberg} $[x-(1-a)]\cdot[x-a]=0$. Then $0=(c_{\eta}-p)\wedge id)_*([x-(1-a)]\cdot[x-a])=(c_{\eta}-p)_*([x-(1-a)])\cdot [x-a]=(\la 1-a\ra-1)\cdot [x-a].$
\end{proof}

\begin{rem}\label{K1}
According to~\cite[Lemma 2.4]{Morel} $K_1^{MW}(k)$ is generated as an abelian group by the symbols $[\eta^m,u_1,\ldots u_{m+1}]$ for $m\geqslant 0,u_i\in k^{\times}$ modulo the following relations:
\begin{itemize}
\item[(1)] (Steinberg relation) $[\eta^m,u_1,\ldots,u_{m+1}]=0$ if $u_i+u_{i+1}=1$ for some $i$
\item[(2)] $[\eta^m,u_1,\ldots,u_{i-1},ab,u_{i+1},\ldots,u_m]=[\eta^m,u_1,\ldots,u_{i-1},a,u_{i+1},\ldots,u_m]+$\\
$+[\eta^m,u_1,\ldots,u_{i-1},b,u_{i+1},\ldots,u_m]+[\eta^{m+1},u_1,\ldots,u_{i-1},a,b,u_{i+1},\ldots,u_m]$
\item[(3)] $[\eta^{m+2},u_1,\ldots,u_{i-1},-1,u_{i+1},\ldots,u_{m+3}]+2[\eta^{m+1},u_1m\ldots,u_{i-1},u_{i+1},\ldots,u_{m+3}]=0$
\end{itemize}
\end{rem}

\begin{lem}\label{K1presentation}
As a $GW(k)$-module, degree one Milnor-Witt K-theory $K_1^{MW}(k)$ is generated by symbols $[a]$ modulo the relations $\langle 1-a\rangle[a]=[a]$, $[ab]=[a]+\langle a\rangle[b]$ for $a,b\in k^{\times}$
\end{lem}
\begin{proof}
Let $A$ denote the $GW(k)$-module generated by the symbols $[a]$ modulo the relations $\langle 1-a\rangle[a]=[a]$, $[ab]=[a]+\langle a\rangle[b]$ for $a,b\in k^{\times}.$ Recall that $GW(k)$ is identified with $K_0^{MW}(k)$ via $\langle a\rangle=\eta[a]+1$ (\cite[\S 2.1]{Morel}). This suggests a mapping $F\colon K_1^{MW}(k)\to A.$
\[F\colon [\eta^m,u_1,\ldots u_{m+1}]\mapsto (\langle u_1\rangle-1)\ldots(\langle u_m\rangle-1)[u_{m+1}]\]

Let us check that $F$ is well defined.
First, note that the relation $[ab]=[a]+\langle a\rangle[b]$ implies that $(\langle a\rangle-1)[b]=(\langle b\rangle-1)[a]$ in $A$. Then for any permutation $\sigma$ of the set $\{u_1,\ldots u_{m+1}\}$ we have $F([\eta^m,\sigma(u_1),\ldots\sigma(u_{m+1})])=F([\eta^m,u_1,\ldots u_{m+1}])$.
Now check that the relations $(1)-(3)$ of Remark~\ref{K1} hold for $F$-images in $A$:

The Steinberg relation $(1):$ consider a symbol $x=[\eta^m,u_1,\ldots,u_{m+1}]$ with $u_j+u_{j+1}=1$ for some $j$. Since $F(x)$ is stable under any permutation of $u_i,$ we may assume that $j=m,$ so $u_m=1-u_{m+1}$. Then $F(x)=0$ since $(\la 1-u_{m+1}\ra-1)[u_{m+1}]=0$ in $A.$

The relation $(2)$: Let $x=[\eta^m,u_1,\ldots,u_{i-1},ab,u_{i+1},\ldots,u_{m+1}].$ Then $F(x)=F([\eta^m,u_1,\ldots,u_{i-1},u_{i+1},\ldots,u_{m+1},ab])=\prod_{j\neq i}(\la u_{j}\ra-1)[ab]=\prod_{j\neq i}(\la u_{j}\ra-1)[a]+\prod_{j\neq i}(\la u_{j}\ra-1)\la a\ra[b]=F([\eta^m,u_1,\ldots,u_{m+1},a])+
F([\eta^m,u_1,\ldots,u_{m+1},b])+$

\noindent$+F([\eta^{m+1},u_1,\ldots,u_{m+1},a,b].$

The relation $(3).$ Denote $y=F([\eta^{m+2},u_1,\ldots,u_{i-1},-1,u_{i+1},\ldots,u_{m+3}])+$

\noindent$+F(2[\eta^{m+1},u_1,\ldots,u_{i-1},u_{i+1},\ldots,u_{m+3}])$. We may assume that $i\neq m+3.$ Denote $q=\prod_{j\neq i}(\la u_j-1\ra-1)\in GW(k).$
Then we have that $y=q(\la -1\ra+1)[u_{m+3}]$ in $A$. Since $m\geqslant 0$, $q$ has at least one factor $(\la u_1\ra-1)$. Since $(\la u_1\ra-1)(\la -1\ra+1)=0$ in $GW(k)$ we get that $y=0$.

Therefore $F\colon K_1^{MW}(k)\to A$ is a well-defined homomorphism of $GW(k)$-modules.
Since the relation $[ab]=[a]+\la a\ra[b]$ holds in $K_1^{MW}(k)$, the mapping $\la a\ra [b]\mapsto [\eta^0,a]+[\eta,a,b]$ from $A$ to $K_1^{MW}(k)$ is a well-defined inverse of $F$, so $F$ is an isomorphism.
\end{proof}

\subsection{Construction of $\Psi$}\label{psiconst}
Now we are ready to construct the map $\Psi\colon K_*^{MW}(k)\to H_0(\ZZ F(\Delta^{\bullet}_k,\Gm^{\wedge *}))$ as follows:
First we construct the $GW(k)$-linear homomorphism
\[\Psi_1\colon K_1^{MW}(k)\to H_0(\ZZ F(\Delta^{\bullet}_k,\Gm^{\wedge 1})), \Psi_1\colon [a]\mapsto [x-a].\] Using the presentation of $K_1^{MW}(k)$ given by~\ref{K1presentation} we see that the relations of $K_1^{MW}(k)$ hold in $H_0(\ZZ F(\Delta^{\bullet}_k,\Gm^{\wedge 1}))$ by Corollary~\ref{Steinbergin1} and Lemma~\ref{K1multiplication}. Then $\Psi_1$ is well defined.
Now we use the presentation of $K_{\geqslant 0}^{MW}(k)$ given by~\cite[Remark 2.2]{Morel}: it is the quotient of the tensor algebra of the $GW(k)$-module $K_1^{MW}(k)$ modulo the Steinberg relation.
\[K_{\geqslant 0}^{MW}(k)=Tens_{GW(k)}K_1^{MW}(k)/([a][1-a]=0)\]
Since the Steinberg relation holds in $H_0(\ZZ F(\Delta^{\bullet}_k,\Gm^{\wedge *}))$ by~\ref{Steinberg}, we get that the map
\[\Psi\colon K_{\geqslant 0}^{MW}(k)\to H_0(\ZZ F(\Delta^{\bullet}_k,\Gm^{\wedge *})), \Psi\colon a_1\otimes\ldots\otimes a_n\mapsto\Psi_1(a_1)\cdot\ldots \Psi_1(a_n)\]
is well defined.

\section{Isomorphism}\label{IsomorphismSection}
In this section we prove the main theorem~\ref{main}.

\begin{lem}\label{taucomp}
Let $L=k(\alpha)$ be a finite extension of $k$. Let $\tau^L_{k}(\alpha)\colon K_1^{MW}(L)\to K_1^{MW}(k)$ be the geometrical transfer defined in~\cite[3.2]{Morel} Then $\tau^L_{k}(\alpha)[\alpha]=\la(-1)^{[L:k]+1}\ra[N(\alpha)].$
\end{lem}
\begin{proof}
Consider a part of the long exact sequence given by~\cite[Theorem 2.24]{Morel} $K_2^{MW}(k(t))\stackrel{\oplus\partial_{(P)}^P}\to\oplus_{P}K_1^{MW}(k[t]/P)$.
Let $p$ be the minimal monic polynomial for $\alpha$.
Then $[p(t)][t]+\la-1\ra[t][p(0)]$ is the preimage of $[\alpha]$ under this map.
Let $n=[L:k]=\deg(p)$
Then by definition we have $\tau^L_k(\alpha)([\alpha])=-\partial_{\infty}^{-1/t}([p(t)][t]+\la-1\ra[t][p(0)])=-\la-1\ra[(-1)^n]+\la-1\ra[p(0)]=\la(-1)^{n+1}\ra[p(0)(-1)^n]=\la(-1)^{n+1}\ra[N(\alpha)].$
\end{proof}

We will use the notation of~\cite[Definition 3.26]{Morel} and for an extension $L=k(\alpha)$ let $\omega_0(\alpha)=p'(\alpha)$ where $p$ is the minimal polynomial of $\alpha$ and $p'$ is its derivative.

\begin{lem}\label{K1decomp} Let $L$ be a finite extension of $k$, $[L:k]=l$ is a prime number and $k$ has no field extensions of degree prime to $l$.
Then $K_1^{MW}(L)$ is generated as abelian group by $GW(L)\cdot K_1^{MW}(k)+A$ where $A=\{\la\pm\omega_0(a)\ra[a]\mid a\in L^{\times}\}$
\end{lem}
\begin{proof}
Consider $\Psi_1\colon K_1^{MW}(L)\to H_0(\ZZ F(\Delta^{\bullet}_L,\Gm^{\wedge 1})).$ It is injective since $\Phi_1\circ\Psi_1=id.$
Take $a\in L^{\times}.$ If $a\notin k$, then $L=k(a)$. Let $p\in k[t]$ be the minimal monic polynomial for $a$. Since $k$ has no prime-to-$l$ extensions, all roots of $p'(x)$ lie in $k.$ Then $p'(x)=x^d(x-\alpha_1)\ldots(x-\alpha_m)$ where $m+d=l-1.$ Let $\alpha=\alpha_1\cdot\ldots\cdot\alpha_m\in k^{\times}$.
Note that $\la x\ra[x]=(\eta[x]+1)[x]=\eta[x][x]+[x]\stackrel{\text{\cite[Lemma 2.7]{Morel}}}=\eta[x][-1]+[x]=\la-1\ra[x]$ and $\la x-1\ra[x]=\la -1\ra[x]$. Then we get that
$\la (a\alpha)^d(a\alpha-1)^m\ra[a\alpha]$ equals $[a\alpha]$ or $\la-1\ra[a\alpha]$ depending on parity of $d$ and $m$

Then $\Psi_1(\la (a\alpha)^d(a\alpha-1)^m\ra[a\alpha])=[x^d(x-1)^m(x-a\alpha)]\stackrel{\ref{deform}}=[p'(x)(x-a)]$.
Finally note that for $c=((\Gm)_L\setminus\{a\alpha\},p'(x)(x-a))$ we have $[p'(x)(x-a)]=[p'(a)(x-a)]+c$.
By Lemma~\ref{ratsupp} we have that $c$ lies in the image of $\Psi_1$ and since all roots of $p'$ lie in $k$ we have that $c\in\Psi_1(GW(L)K_1^{MW}(k)).$

Since $\Psi_1$ is injective, we get $\la\pm 1\ra[a\alpha]=\la p'(a)\ra[a\alpha]+\Phi_1(c).$ Thus $\la\pm 1\ra([a]+\la a\ra[\alpha])=\la p'(a)\ra[a]+\la p'(a)a\ra[\alpha]+\Phi_1(c).$ Since $\Phi_1\in GW(L)K_1^{MW}(k)$. So we expressed $\la\pm 1\ra[a]$ as a sum of elements in $A$ and $GW(L)\cdot K_1^{MW}(k).$
\end{proof}

\begin{lem}\label{MWcomm}
Suppose $F/k$ is a finite field extension. Then $F\otimes_k F=\coprod F_i$ for some fields $F_i$. Then the following diagram commutes in Milnor-Witt $K$-theory:
\[
\xymatrix{
\coprod K_*^{MW}(F_i)\ar[rr]^{\sum_i Tr^{F_i}_F} && K_*^{MW}(F)\\
K_*^{MW}(F)\ar[u]^{diag}\ar[rr]^{Tr^F_k} && K_*^{MW}(k)\ar[u]
}\eqno{(*)}
\]
\end{lem}
\begin{proof}
Since $F/k$ is separable, for every monic irreducible $p_1\in k[t]$ and monic irreducible $p\in F[t]$ such that the point $(p)\in\A^1_F$ lies over $(p_1)\in\A^1_k$ (we will write $p|_{p_1}$ in this case) we have that the residue $\overline{p_1/p}$ is invertible in the field $F[t]/(p).$ Then by~\cite[Axiom B3]{Morel} the following diagram commutes:
\[
\xymatrix{
K_*^{MW}(F) & K_{*+1}^{MW}(F(t))\ar[l]_{-\partial_{\infty}^{-1/t}}\ar[r]^{\oplus \partial_{(p)}^p} & \oplus_{p} K_*^{MW}(F[t]/p)\\
K_*^{MW}(k)\ar[u]^{j} & K_{*+1}^{MW}(k(t))\ar[l]_{-\partial_{\infty}^{-1/t}}\ar[r]^{\oplus \partial_{(p_1)}^{p_1}}\ar[u]^{j} & \oplus_{p_1} K_*^{MW}(F[t]/p_1)\ar[u]_{\sum_{p|_{p_1}}\la \overline{p_1/p}\ra\cdot j}
}
\]
where $j$ denotes the pullback map for the field extension. Now take a generator $\alpha$ of $F$ over $k$ and $p$ its minimal polynomial. This gives embedding $\Spec F\to\A^1_k.$ Taking the fiber product with $\A^1_F\to\A^1_k$ we get an embedding $\coprod\Spec F_i=\Spec F\otimes_k F\to\A^1_F$. This gives a choice of generators $\alpha_i$ of $F_i$. Let $p_i$ be the monic irreducible polynomials for $\alpha_i.$ This implies the commutative diagram for geometric transfers:
\[
\xymatrix{
K_*^{MW}(F) & \coprod K_*^{MW}(F_i)\ar[l]_{\sum \tau^{F_i}_F(\alpha_i)}\\
K_*^{MW}(k)\ar[u]^{j} & K_*^{MW}(F)\ar[l]_{\tau^{F}_k(\alpha)}\ar[u]_{\sum_i \la\overline{p/p_i}\ra\cdot j}
}
\]
Since $\overline{p/p_i}=p'(\alpha_i)/(p_i'(\alpha_i))$ the latter diagram implies the commutativity of the diagram with canonical transfers $(*).$
\end{proof}

So it is sufficient to check that the map $\Psi$ is surjective. We will need the following analogue of~\cite[Lemma 5.11]{Voevodsky}:

\begin{lem}\label{Maincomm}
 Let $L/k$ be a finite field extension. There is a transfer map $Tr^L_k\colon K_n^{MW}(L)\to K_n^{MW}(k)$ given by~\cite[Definition 3.26]{Morel} and~\cite[Theorem 3.27]{Morel}. Then the diagram commutes:
\[
\xymatrix{
K_n^{MW}(L)\ar[rr]^{\Psi_L}\ar[d]^{Tr^L_k} && H_0(\ZZ F(\Delta^{\bullet}_L,\Gm^{\wedge n}))\ar[d]^{tr_{L/k}}\\
K_n^{MW}(k)\ar[rr]^{\Psi} && H_0(\ZZ F(\Delta^{\bullet}_k,\Gm^{\wedge n}))
}
\]
\end{lem}
\begin{proof}
We use the same strategy as~\cite[Lemma 5.11]{Voevodsky}.

First, consider the case when $L:k=l$ is prime and $k$ has no extensions of degree prime to $l$. Then by~\cite[Lemma 2.25]{Morel} we have that $K_n^{MW}(L)$ is generated by the elements of the form $\eta^m[a_1]\ldots[a_{n+m}]$ where $a_2,\ldots,a_{n+m}\in k^{\times}.$ Then, the projection formulas for $Tr^L_k$ and $tr_{L/k}$ imply that it is sufficient to check the commutativity of the diagram for $n=0,1.$ The case $n=0$ is given by~\ref{com0}.

For the case $n=1$ note that for any $a\in L^{\times}$ with minimal polynomial $p$ over $k$ we have
$tr_{L/k}(\Psi_1(\la p'(a)\ra[a]))=tr_{L/k}(\la p'(a)\ra[x-a])\stackrel{\ref{tracegen}}=\la(N(a)-1)^{l-1}\ra[x-N(a)]=\la(-1)^{l-1}\ra\la(1-N(a))^{l-1}\ra[x-N(a)]\stackrel{\ref{Steinbergin1}}=\la(-1)^{l-1}\ra[x-N(a)]=
\Psi_1(\la(-1)^{l-1}\ra[N(a)])\stackrel{\ref{taucomp}}=\Psi_1(\tau^L_k(a)([a]))=
\Psi_1(Tr^L_k(\la p'(a)\ra[a])).$

For any $a\in L^{\times}$ and $b\in k^{\times}$ we have $tr_{L/k}(\Psi_1(\la a\ra[b]))=tr_{L/k}(\la a \ra[x-b])\stackrel{\ref{projection}}=tr_{L/k}(\la a\ra)[x-b]\stackrel{\ref{com0}}=\Psi_1(Tr^L_k(\la a\ra)[b]).$ Then the diagram commutes by~\ref{K1decomp}.

In the general case we use the standard reduction. Let $L'$ be the maximal prime-to-$l$ field extension of $k$.
Denote $H_0(\ZZ F(\Delta^{\bullet}_L,\Gm^{\wedge n}))$ by $H^{n}(L)$. Then for any $x$ in the kernel of $H^{n}(k)\to H^{n}(L')$
we have that $m_{\epsilon}x=0$ for some $m$ prime to $l$ by~\ref{transfer1}.
Let $t=tr_{L/k}(\Psi(a))-\Psi(Tr^L_k(a)).$ Then $t_{L'}=0$ by the previous case, so $m_{\epsilon}t=0$ for some $m$ prime to $l$. Then
since $(m_{\epsilon},l_{\epsilon})=1$ in $GW(k)$, it is sufficient to check that $t_L=0.$
Note that $L\otimes_k L=\coprod L_i$ where $L_i$ are field extensions of $L$ with $[L_i:L]<[L:k].$ The two diagrams commute:
\[
\xymatrix{
K_n^{MW}(L)\ar[r]\ar[d]^{Tr^L_k} & \oplus K_n^{MW}(L_i)\ar[d]^{\sum Tr^{L_i}_L} &&& H^n(L)\ar[r]\ar[d]^{tr_{L/k}} & \oplus H^n(L_i)\ar[d]^{\sum tr_{L_i/L}}\\
K_n^{MW}(k)\ar[r] & K_n^{MW}(L)                                                   &&& H^n(k)\ar[r] & H^n(L)
}
\]
The left diagram commutes by~\ref{MWcomm} To show the commutativity of the right diagram check that $c_{L/k}\circ pr_k=\sum_i pr_L\circ c_{L_i/L}$ in $H_0(\ZZ F(\Delta^{\bullet}_L,\Spec L)).$ Choose a generator $\alpha$ of $L$. Then $c_{L/k}\circ pr_k=(L\otimes_k\A^1_L,(x-\alpha)\circ pr_{\A^1_L},pr_L)=(\coprod \A^1_{L_i},\coprod x-\alpha_i,pr_L)=\sum_i pr_L\circ c_{L_i/L},$ where $\alpha_i$ are the generators of $L_i$ given by the pullback of the diagram
\[\Spec L\stackrel{\alpha}\to\A^1_k\leftarrow\A^1_L.\]

By induction on $[L:k]$, the diagram commutes for all extensions $L_i/L$, so we have
\[t_L=\sum_i tr_{L_i/L}(\Psi(a))-\Psi(Tr^{L_i}_L(a))=0.\] So $t_L$ is zero, then $t=0,$ so the diagram is commutative.
\end{proof}

\begin{prop}\label{Psisurj}
The map $\Psi\colon K_*^{MW}(k)\to H_0(\ZZ F(\Delta^{\bullet}_k,\Gm^{\wedge *}))$ is surjective.
\end{prop}
\begin{proof}
By lemma~\ref{red1} the group $H_0(\ZZ F(\Delta^{\bullet}_k,\Gm^{\wedge n}))$ is generated by the correspondences of the form $c=(\A^1_L,\mu(x-\lambda),f(Z)\circ pr_L)$ where $L=k(\lambda)$ and $f(Z)\in \Gm^{n}(L).$ Take a correspondence $c'\in Fr(L,\Gm^n)$ represented by the data $c'=(\A^1_L,\mu(x-\lambda),f(Z)\circ pr_L).$ Then $tr_{L/k}(c')$ is given by the data $(\A^2_L,\mu(x-\lambda),x_2-\lambda,f(Z)\circ pr_L)\stackrel{\ref{addone}}\sim (\A^1_L,\mu(x-\lambda),f(Z)\circ pr_L)=c.$

Lemma~\ref{ratsupp} implies that $c'$ is equivalent to a standard correspondence, thus $c'$ lies in the image of $\Psi_L\colon K_n^{MW}(L)\to H_0(\ZZ F(\Delta^{\bullet}_L,\Gm^{\wedge n}))$. Then by Lemma~\ref{Maincomm} $c=tr_{L/k}(c')$ lies in the image of $\Psi.$
\end{proof}

This immediately implies the main theorem:
\begin{thm}\label{main}
The maps $\Phi$ and $\Psi$ are mutually inverse ring isomorphisms between $K_{*\geqslant 0}^{MW}(k)$ and $H_0(\ZZ F(\Delta^{\bullet}_k,\Gm^{\wedge *}))$
\end{thm}
\begin{proof}
By Lemma~\ref{Phiongen} and construction on $\Psi$ we have that $\Phi\circ\Psi=id_{K^{MW}(k)}$. Since $\Psi$ is surjective by~\ref{Psisurj}, this implies that $\Psi$ and $\Phi$ are mutually inverse isomorphisms of groups. By its construction, $\Psi$ is a ring homomorphism, so $\Phi$ is also a ring homomorphism.
\end{proof}

\bibliographystyle{plain}

\begin{thebibliography}{10}

\bibitem{BW} P. Balmer and C. Walter.
A Gersten-Witt spectral sequence for regular schemes.
{\em Ann. Sci. École Norm. Sup.} (4) 35 (2002), no. 1, 127--152.

\bibitem{Fasel} J.~Fasel.
The Chow-Witt ring.
{\em Doc. Math.} 12 (2007), 275--312.

\bibitem{FaselThesis} J.~Fasel.
Groupes de Chow-Witt.
{\em Mém. Soc. Math. Fr. (N.S.)} No. 113 (2008).

\bibitem{GPFramedMotives} G.~Garkusha and I.~Panin.
Framed Motives of algebraic varieties (after V. Voevodsky),
{\em Preprint}, arXiv:1409.4372

\bibitem{GPLinear} G.~Garkusha and I.~Panin.
The triangulated category of linear framed motives $DM^{eff}_{fr}(k)$, in preparation.

\bibitem{Gille} S.~Gille.
A transfer morphism for Witt groups.
{\em J. Reine Angew. Math.} 564 (2003), 215--233.

\bibitem{GN} S. Gille and A. Nenashev.
Pairings in triangular Witt theory.
{\em J. Algebra} 261 (2003), no. 2, 292--309.

\bibitem{Hartshorne}
R. Hartshorne, \emph{Algebraic geometry}, Graduate Texts in Mathematics, No. 52. Springer-Verlag, New York-Heidelberg, 1977.

\bibitem{Voevodsky}
C. Mazza, V.Voevodsky and C. Weibel, \emph{Lecture notes on motivic cohomology}, Clay Mathematics Monographs, 2. American Mathematical Society, Providence, RI; Clay Mathematics Institute, Cambridge, MA, 2006.

\bibitem{Morel}
F. Morel, \emph{$\A^1$-algebraic topology over a field}, Lecture Notes in Mathematics, 2052. Springer, Heidelberg, 2012.

\bibitem{NenashevTransfer}
A. Nenashev.
Projective push-forwards in the Witt theory of algebraic varieties.
{\em Adv. Math.} 220 (2009), no. 6, 1923--1944.

\bibitem{NenashevGysin}
A. Nenashev.
Gysin maps in Balmer-Witt theory.
{\em J. Pure Appl. Algebra} 211 (2007), no. 1, 203--221.

\bibitem{VoevUnpublished}
V. Voevodsky.
Notes on framed correspondences.
{\em unpublished}, (2001--2003).
\end{thebibliography}

\end{document}